\documentclass[draft,12pt]{amsart}
\usepackage{url}

\newtheorem{theorem}{Theorem}[section]

\newtheorem{prop}[theorem]{Proposition}

\theoremstyle{remark}
\newtheorem{remark}[theorem]{Remark}

\theoremstyle{definition}
\newtheorem{defn}[theorem]{Definition}
\newtheorem{example}[theorem]{Example}

\numberwithin{equation}{section}
\numberwithin{theorem}{section}

\newcommand{\sgn}{{\rm sgn}}
\newcommand{\balex}{{\mathcal B}_c(\Rbar)}
\newcommand{\balextwo}{{\mathcal B}_c(\Rbar^2)}
\newcommand{\balexn}{{\mathcal B}_c(\Rbar^n)}
\newcommand{\alex}{{\mathcal A}_c(\R)}
\newcommand{\alextwo}{{\mathcal A}_c(\R^2)}
\newcommand{\alexn}{{\mathcal A}_c(\R^n)}
\newcommand{\supp}{{\rm supp}}
\newcommand{\Rbar}{\overline{\R}}
\newcommand{\D}{{\mathcal D}(\R)}
\newcommand{\Dtwo}{{\mathcal D}(\R^2)}

\newcommand{\Dp}{{\mathcal D}'(\R)}
\newcommand{\Dptwo}{{\mathcal D}'(\R^2)}
\newcommand{\Dpn}{{\mathcal D}'(\R^n)}
\newcommand{\dalex}{\partial_{12}}

\newcommand{\bv}{\mathcal{BV}}
\newcommand{\hkbv}{\mathcal{HKBV}}
\newcommand{\intinf}{\int^\infty_{-\infty}}
\newcommand{\intxy}{\int_{-\infty}^x\int_{-\infty}^y}
\newcommand{\intabcd}{\int_a^b\int_c^d}

\newcommand{\Lone}{L^1(\R^2)}
\newcommand{\Q}{{\mathbb Q}}
\newcommand{\N}{{\mathbb N}}
\newcommand{\R}{{\mathbb R}}
\newcommand{\C}{{\mathbb C}}

\newcommand{\fn}{\!:\!}

\newcommand{\nbv}{\mathcal{NBV}}

\providecommand{\abs}[1]{\lvert#1\rvert}
\providecommand{\norm}[1]{\lVert#1\rVert}

\begin{document}

\date{Preprint April 4, 2020.  To appear in {\it Real Analysis Exchange}.}
\title[Continuous primitive integral]
{The continuous primitive integral in the plane}
\author{Erik Talvila}
\address{Department of Mathematics \& Statistics\\
University of the Fraser Valley\\
Abbotsford, BC Canada V2S 7M8}
\email{Erik.Talvila@ufv.ca}

\begin{abstract}
An integral is defined on the plane that includes the Henstock--Kurzweil and
Lebesgue integrals (with respect to Lebesgue measure).  A space of primitives
is taken as the set of continuous real-valued functions $F(x,y)$ defined on 
the extended real plane $[-\infty,\infty]^2$ that vanish when $x$ or $y$ is $-\infty$.  
With usual pointwise operations
this is a Banach space under the uniform norm.  The integrable functions and
distributions (generalised functions) are those that are the distributional
derivative $\partial^2/(\partial x\partial y)$ of this space of primitives.  If
$f=\partial^2/(\partial x\partial y) F$ then the integral over interval $[a,b]\times [c,d]
\subseteq[-\infty,\infty]^2$ is $\int_a^b\int_c^d f=F(a,c)+F(b,d)-F(a,d)-F(b,c)$ and
$\int_{-\infty}^\infty \int_{-\infty}^\infty f=F(\infty,\infty)$.
The definition then builds in the fundamental theorem of calculus.  The Alexiewicz
norm is ${\lVert f\rVert}={\lVert F\rVert}_\infty$ where $F$ is the unique primitive of $f$.  The
space of integrable distributions is then a separable Banach space isometrically 
isomorphic to the space of primitives.  The space of integrable distributions is
the completion of both $L^1$ and the space of Henstock--Kurzweil
integrable functions.  The Banach lattice and Banach algebra
structures of the continuous functions in ${\lVert \cdot\rVert}_\infty$ are also inherited
by the integrable distributions.  It is shown that the dual space are the functions
of bounded Hardy--Krause variation.  Various tools that make these integrals useful
in applications are proved: integration by parts, H\"older inequality, second mean value theorem, Fubini
theorem, a convergence theorem, change of variables, convolution.  The changes necessary to define the integral
in ${\mathbb R}^n$ are sketched out.
\end{abstract}

\maketitle

\section{Introduction}\label{sectionintroduction}
The continuous primitive integral is discussed in $\R^2$ and then briefly in $\R^n$.
This is an integral defined by taking primitives (indefinite integrals) as continuous
functions.  It includes the Lebesgue and Henstock--Kurzweil integrals.  The essential
idea is to take a Banach space ${\mathcal B}$ of primitives and define the entities that can be
integrated as the distributional derivative of each item in ${\mathcal B}$.  Here ${\mathcal B}$
is taken as the continuous functions on the extended real plane.  Each such function is
differentiated with the partial differential operator $\dalex=\partial^2/(\partial y\partial x)$.  
This automatically makes
the distributions integrable in this sense into a Banach space isometrically isomorphic
to the continuous functions under the uniform norm.  

The same process can be repeated with
other classes of primitives.  There is the regulated primitive integral \cite{talvilaregulated}.
A function on the real line is regulated if it has a left and right limit at each point, or
from within each orthant in $\R^n$.  There is the $L^p$ primitive integral \cite{talvilaLp}.
And there are higher order distributional integrals for which each continuous function is
differentiated multiple times \cite{talvilaacrn}.

The name, continuous primitive integral, was introduced at the end of \cite{talviladenjoy}.
Some authors refer to the same integral as the distributional Henstock--Kurzweil
or distributional Denjoy integral.  As there are several integrals defined by their primitives,
as above, we prefer the name continuous primitive integral.

First define the primitives.
The extended real line is $\Rbar=[-\infty,\infty]$.  A function $F\fn \Rbar\to\R$ is continuous on
$\Rbar$
if it equals its limit at each point, $F(x)=\lim_{t\to x}F(t)$, where the limit is necessarily
one-sided if $x=\infty$ or $-\infty$.  The extended real plane is $\Rbar^2$ endowed with
the product topology.  We then take as a space of primitives $\balextwo$ which consists of
the continuous functions $F(x,y)$ on $\Rbar^2$ that vanish when $x=-\infty$ or $y=-\infty$.  
Under the uniform norm $\balextwo$ is a Banach space.  A distribution (generalised function), $f$,
has a continuous primitive integral if there is a function $F\in\balextwo$ such that $f=\dalex F$,
the partial derivative being understood in the distributional sense.  Since $F(x,y)=0$ if $x$ or $y$
is $-\infty$, the primitive is unique.  If $(x,y)\in\Rbar^2$
then the integral is $\int_{-\infty}^x\int_{-\infty}^y f=F(x,y)$, with a similar definition on
compact intervals.  In this way the definition builds in the fundamental theorem of calculus.

The Alexiewicz norm of $f$ is 
$$
\norm{f}=\sup_{(x,y)\in\Rbar^2}\left|\int_{-\infty}^x\int_{-\infty}^y f\right|=\norm{F}_\infty. 
$$
Write the set of integrable distributions as $\alextwo$.  Then $\alextwo$
is a Banach space that is isometrically
isomorphic to $\balextwo$.  Since the Lebesgue and Henstock--Kurzweil integrals have continuous
primitives they form dense subspaces of $\alextwo$ but neither is complete in this norm.
The continuous primitive integral then provides the completion with respect to the Alexiewicz norm
of the space of Henstock--Kurzweil integrable
functions.  The Henstock--Kurzweil integral allows conditional convergence and so does the continuous
primitive integral.

The Henstock--Kurzweil integral is a well-established integration process based on Riemann sums
that includes the
Lebesgue and improper Riemann integrals in $\R^n$ (with respect to Lebesgue measure).  For early
results see \cite{jarnikkurzweil}, \cite{mawhin1981}, \cite{pfeffer} and \cite{kurzweilmawhinpfeffer}.  It
is discussed on the real line and briefly in $\R^2$ or $\R^n$ in the monographs \cite{mawhin}, 
\cite{mcleod}, \cite{swartz} and
\cite{yee}.
A detailed treatment of the Henstock--Kurzweil integral on compact intervals in $\R^n$ is
given in \cite{bongiornopap} and \cite{lee}, where there is also an extensive review of the literature.
See also \cite{kurzweil}.
The Denjoy integral is equivalent to the Henstock--Kurzweil integral and is defined 
via properties of the primitive.
See \cite{celidze}.

Under the usual pointwise operations,
$\balextwo$ is a Banach lattice and Banach algebra; and $\alextwo$ inherits these properties.

The simple structure of $\balextwo$ makes it easy to prove various results in $\alextwo$.  
The corresponding space of primitives
for the Lebesgue integral are the absolutely continuous functions.  There are many different
notions of absolute continuity for functions of two variables, due to Tonelli and
other authors.
If $f\in L^1(\R^2)$ and $F(x,y)=\int_{-\infty}^x\int_{-\infty}^y f$
then $F$ is absolutely continuous in the sense of Carath\'{e}odory.  See \cite{sremr} for the
definition and references to Carath\'{e}odory's original work. The primitives for the 
Henstock--Kurzweil integral in $\Rbar^2$
are much more complicated than $\balextwo$.  See \cite{celidze}.  The primitives for Lebesgue
and Henstock--Kurzweil integrals are continuous and the pointwise derivative $\dalex$ exists
almost everywhere.  Being merely continuous, primitives in $\balextwo$ need not have a pointwise
derivative anywhere but the distributional derivative $\balextwo$ is well-defined. See following
Definition~\ref{defnalex}.

There are many different
notions of bounded variation for functions of two variables
(\cite{clarksonadams}, \cite{adamsclarkson},
\cite{adamsclarksoncorrection}).  If $g$ is of bounded Hardy--Krause variation then the product
$fg$ is in $\alextwo$ for all $f\in\alextwo$ and we can prove an integration by parts formula and
H\"older inequality.  Functions of bounded Hardy--Krause variation also form the dual space of
$\alextwo$.

The paper is laid out as follows.

Section~\ref{sectiondistributions} gives the necessary background in distributions.

Functions on the extended real plane are discussed in Section~\ref{sectionRbar2}.

In Section~\ref{sectioncontinuousprimitiveintegral} the continuous primitive integral is
defined on intervals in $\Rbar^2$ and various basic properties, such as linearity and
the fundamental theorem of calculus, are proved.  
It is shown that $\alextwo$ is a separable
Banach space isometrically isomorphic to the space of primitives $\balextwo$.  The test
functions, the real analytic functions, $L^1$ and the Henstock--Kurzweil integrable functions
are all shown to be dense in $\alextwo$.  It is shown that the integral can be defined as
the limit of a sequence of Lebesgue integrals.

Various examples are given in Section~\ref{sectionexamples}.  We have already noted above
that the continuous primitive integral includes the Lebesgue and Henstock--Kurzweil integrals.
If $F\in\balextwo$ and $f=\dalex F$ then
an example of note is  the case when the primitive $F$ has a pointwise derivative
$\dalex F$ nowhere.
Then $\intabcd f$ is well-defined in $\alextwo$ but
the Lebesgue integral of $f$ does not exist.  Also, if $\dalex F=0$ almost everywhere then the
Lebesgue integral of $f$ is $0$ over every interval but the continuous primitive integral
gives the value we would expect from the fundamental theorem of calculus.
In this section we also discuss other compactifications of $\R^2$.

Functions of Hardy--Krause bounded variation are defined in Section~\ref{sectionhkbv} and
some examples are given.

In Section~\ref{sectionparts} it is shown that the functions of Hardy--Krause bounded 
variation form
the multipliers and allow us to prove an integration by parts formula in terms of Henstock--Stieltjes
integrals.  This leads to versions of the first and second mean value theorems for integrals.
It is shown that $\alextwo$ is invariant under translations and that translations are 
continuous in the Alexiewicz norm.

A type of H\"older inequality is proved in Section~\ref{sectionholder}.  This gives the
inequality $\abs{\intinf\intinf fg}\leq \norm{f}\norm{g}_{bv}$ for $f\in\alextwo$ and $g$
of Hardy--Krause bounded variation.  Some norms equivalent to $\norm{\cdot}$ are introduced.
It is shown that the dual space of $\alextwo$ is the space of functions of 
Hardy--Krause bounded variation.

A convergence theorem is given in Section~\ref{sectionconvergence} for taking the limit
under integrals $\intinf\intinf fg_n$ where $f\in\alextwo$ and $g_n$ is a sequence of
functions of bounded Hardy--Krause variation.

If $f\in\alextwo$ then, in general, the only subsets $f$ is integrable on are finite
unions of intervals in $\Rbar^2$.  Hence, a change of variables theorem can only map
intervals to finite unions of intervals.  In Section~\ref{sectionchange} a change of
variables theorem is given where each variable $(x,y)$ is transformed to a linear
combination of just one variable.

In Section~\ref{sectionbanachlattice} a partial ordering is introduced on $\alextwo$ that
makes this into a Banach lattice isomorphic to $\balextwo$ under the usual pointwise ordering.
Both $\balextwo$ and $\alextwo$ are abstract $M$-spaces.

In Section~\ref{sectionbanachalgebra} the pointwise algebra structure on $\balextwo$,
defined as usual by $(FG)(x,y)=F(x,y)G(x,y)$, is extended
to $\alextwo$ so that it becomes a
Banach algebra, without a unit but with an approximate identity, isomorphic to $\balextwo$.  

A sufficient condition for changing the order of iterated integrals is given in
Section~\ref{sectioniteratedintegrals}.  Some examples are given for which iterated integrals
are not equal.  Examples of this type can be resolved by showing the primitive is not
continuous on the closure of the interval of integration, although it may be continuous
on the interior of the interval of integration.

Convolutions $f\ast g$ are defined in Section~\ref{sectionconvolution} for $f\in\alextwo$
and $g$ of Hardy--Krause bounded variation.  These behave similarly to convolutions when
$f\in L^1$ and $g\in L^\infty$.  Convolutions are also defined for $g\in L^1(\R^2)$ and
these behave similarly to convolutions when $f,g\in L^1$.

Finally, some of the changes needed to define the integral in $\Rbar^n$ are sketched out
in Section~\ref{sectionRbarn}.

The notion of using continuous functions for primitives appears to have first been 
considered by K.~Ostaszewski in \cite{ostaszewski1}.  Then the definition of
the integral was sketched out in the setting of compact intervals 
in $\R^n$ by P.~Mikusinski and K.~Ostaszewski
in \cite{pmikusinski1988} and \cite{pmikusinski1990}.  In the context of the real line 
it was also mentioned briefly by
B.~Bongiorno \cite{bongiornomathstudy}; B.~Bongiorno and T.V. Panchapagesan \cite{bongiornopanchapagesan};
B. B\"aumer, G. Lumer and F. Neubrander \cite{baumerlumerneubrander}.  
It was studied in more detail on compact intervals in $\R^2$ by D.D.~Ang, K.~Schmidt and L.K.~Vy in \cite{ang} 
(with some results repeated in \cite{angvy})
and (on the real line) by E.~Talvila \cite{talviladenjoy}.

The integral was applied to Fourier series \cite{talvilafourierseries} and a type of
Salem--Zygmund--Rudin--Cohen factorization was proved there.  See also \cite{moralesarredondo}.

Various other properties were studied in 
\cite{escamilla},
\cite{gutierrez},
\cite{gutierrez2}, 
\cite{talvilaconv}.

A number of our results are generalisations of similar results proved for the Henstock--Kurzweil
integral in
\cite{lee}.

\section{Distributions}\label{sectiondistributions}
Here we briefly describe notation and a few of the major properties of distributions that we 
will use.
All of the results in distributions we use can be found in
\cite{folland} and \cite{friedlanderjoshi}.

The support of a function $\phi\fn\R^2\to\R$ is the closure of the set on which it does not vanish,
denoted $\supp(\phi)$.
The test functions are $\Dtwo=C_c^\infty(\R^2)=\{\phi\fn\R^2\to\R\mid \phi\in C^\infty(\R^2) \text{ with
compact support}\}$.  Note that $\Dtwo$ is a linear space closed under differentiation.  
If $\{\phi_n\}$ is a sequence of functions in $\Dtwo$ and $\phi\in\Dtwo$ then
$\phi_n\to\phi$ if there is a compact set $K\subset\R^2$ such that for each $n\in\N$ we have
$\supp(\phi_n)\subseteq K$ and for all integers $k,\ell\geq 0$ we have 
$\norm{\partial_1^k\partial_2^\ell\phi_n-\partial_1^k\partial_2^\ell\phi}_\infty
\to 0$ as $n\to\infty$, i.e., all partial derivatives converge uniformly to $\phi$.
The symbol $\partial_i$ represents the partial derivative with respect to the $i$th Cartesian variable.

The distributions are the continuous linear functionals on $\Dtwo$.  This is the dual space of
$\Dtwo$, written $\Dptwo$.  For $T\in\Dptwo$ its action on test function $\phi$ is written
as $\langle T,\phi\rangle\in\R$.  Distributions are linear:   
$\langle T, a\phi+b\psi\rangle =a\langle T,\phi\rangle +b
\langle T,\psi\rangle$ for all $\phi,\psi\in\Dtwo$ and all $a,b\in\R$.  Distributions are continuous: 
if $\phi_n\to\phi$ in $\Dtwo$ then $\langle T,\phi_n\rangle\to
\langle T,\phi\rangle$ in $\R$.  To define distributions on an open set $\Omega\subset\R^2$ we use test
functions with compact support in $\Omega$.

All distributions have derivatives of all orders and all such derivatives are distributions.
For each $i=1,2$ the derivative of $T\in\Dptwo$ is $\langle \partial_i T,\phi\rangle
=-\langle T,\partial_i\phi\rangle$ for each $\phi\in\Dtwo$.
Write $\dalex=\partial_{1}\partial_2$.  Then $\langle \dalex T,\phi\rangle=\langle T,\dalex\phi\rangle$.
All Cartesian derivative operators commute on test functions and distributions.

\section{Extended real plane}\label{sectionRbar2}
The extended real line is $\Rbar=[-\infty,\infty]$.  It is a compact topological space with
a topological base given by usual open intervals in $\R$ together with intervals $[-\infty,a)$,
$(a,\infty]$ for all $a\in\R$.  This is then a two-point compactification of $\R$.  A
function $F\fn\Rbar\to\R$ is continuous at $x\in\R$ if $\lim_{y\to x}F(y)=F(x)$, continuous at
$-\infty$ if $\lim_{y\to -\infty}F(y)=F(x)$, continuous at
$\infty$ if $\lim_{y\to \infty}F(y)=F(x)$.  The last two limits are necessarily one-sided.
For example, the function $\arctan$ is continuous on $\Rbar$ if we define $\arctan(\pm\infty)
=\pm\pi/2$ and no definition at $\pm\infty$ can make the functions $\sin$ or $\exp$ continuous
on $\Rbar$.

The extended real plane is $\Rbar^2$ and has the product topology.  It is then a compact Hausdorff
space.  The continuous functions on $\Rbar^2$ are denoted $C(\Rbar^2)$.  Note that they are real-valued.  We define 
\begin{defn}\label{defnbalex}
\begin{eqnarray*}
\balex & = & \{F\fn\Rbar\to\R\mid
F \text{ is continuous on } \Rbar, F(-\infty)= 0\}\\
\balextwo & = & \{F\fn\Rbar^2\to\R\mid \\
 & & \quad F \text{ is continuous on } \Rbar^2, F(-\infty,s)=
F(s,-\infty)=0 \text{ for all } s\in \Rbar\}.
\end{eqnarray*}
\end{defn}
Hence, a function $F\in\balextwo$ is continuous at $(x,y)\in\R^2$ if for each $\epsilon>0$ there is
$\delta>0$ such that if $(x-\xi)^2 + (y-\eta)^2< \delta^2$ then $\abs{F(x,y)-F(\xi,\eta)}<\epsilon$.
If $x\in\R$ then $F$ is continuous at $(x,\infty)\in\Rbar^2$ if when $\abs{x-\xi}<\delta$ and
$\eta>1/\delta$ we have $\abs{F(x,y)-F(\xi,\eta)}<\epsilon$.  Similarly for other points in 
$\Rbar^2\setminus \R^2$.  Continuity in $\Rbar^2$ implies uniform continuity in $\R^2$ but uniform continuity
in $\R^2$ does not imply continuity or boundedness in $\Rbar^2$.
With the uniform norm, $\norm{\cdot}_\infty$, $\balextwo$ is a Banach space.
Note that if $F\in C(\Rbar^2)$ then
$$
\norm{F}_\infty=\sup_{(x,y)\in\Rbar^2}\abs{F(x,y)}=\sup_{(x,y)\in\R^2}\abs{F(x,y)}=\max_{(x,y)\in\Rbar^2}\abs{F(x,y)}.
$$

\section{The continuous primitive integral}\label{sectioncontinuousprimitiveintegral}
We can now define the integrable distributions as the derivatives of functions in $\balextwo$.
Parts of Propositions~\ref{propunique}, \ref{propseparable}, \ref{propdense} were proved
for compact intervals in \cite{ang}.
\begin{defn}\label{defnalex}
$$
\alextwo=\{f\in\Dptwo\mid f=\dalex F \text{ for some } F\in\balextwo\}.
$$
\end{defn}
In this definition the function $F$ is called the primitive of $f$.  If $f\in\alextwo$ has primitive
$F\in\balextwo$ then the action of $f$ on test function $\phi$ is 
$\langle f,\phi\rangle=\langle F, \dalex\phi\rangle=\intinf\intinf F(x,y)\dalex\phi(x,y)\,dy\,dx$.  Since
$\phi$ is smooth with compact support this last integral exists in the Riemann sense.

If $F$ is a function in $\balextwo$ then $F(x,y)$ vanishes
when $x$ or $y$ is $-\infty$.   Primitives are then unique.  The derivative
operator $\dalex$ is a linear isomorphism between $\alextwo$ and $\balextwo$.   We define its inverse to be the
integral and then $\alextwo$ inherits the Banach space structure of $\balextwo$.
\begin{prop}\label{propunique}
(a) If $f\in\alextwo$ then it has a unique primitive in $\balextwo$.
(b) If $f\in\alextwo$ with primitive $F\in\balextwo$ then define the Alexiewicz norm of $f$ by
$\norm{f}=\norm{F}_\infty$.  Then $\alextwo$ is a Banach space.  The derivative $\dalex$ provides
a linear isometry and isomorphism between 
$\balextwo$ and $\alextwo$.
(c) If $G$ is continuous on $\Rbar^2$ then $\dalex G\in\alextwo$.
(d) For all $f,g\in\alextwo$; $c_1,c_2\in\R$; $\phi\in\Dtwo$ we have
$\langle c_1f+c_2g,\phi\rangle=c_1\langle f,\phi\rangle +c_2\langle g,\phi
\rangle$.
\end{prop}
The Alexiewicz norm first appears in \cite{alexiewicz}.  
\begin{proof}
(a) Suppose $f\in\alextwo$ and $f=\dalex F=\dalex G$ for $F,G\in\balextwo$.  Let $\Phi=F-G$.
Then $\Phi\in\balextwo$ and $\dalex \Phi=0$.  But then $\Phi(x,y)=\Theta(x)+\Psi(y)$ for
some functions $\Theta,\Psi\in C(\Rbar)$.  Fixing $x\in\R$ and letting $y\to-\infty$ and
then fixing $y\in\R$ and letting $x\to-\infty$ shows
$\Theta$ and $\Psi$ are constant functions with sum $0$.
(b) The derivative operator $\dalex$ is linear.  By (a) it is one-to-one on $\balextwo$.  By
definition it is onto $\alextwo$.  The definition of $\norm{\cdot}$ makes it into an isometry.
(c) If $G\in C(\Rbar^2)$ define $\Theta,\Psi\in C(\Rbar)$ by $\Theta(x)=G(x,-\infty)$ and
$\Psi(y)=G(-\infty,y)$.  Define
$F\in \balextwo$ by $F(x,y)=G(x,y)+G(-\infty,-\infty)-\Theta(x)-\Psi(y)$.  Then $\dalex G=\dalex F$.
(d) The derivative $\dalex$ is linear.
\end{proof}

We can now define the integral of a distribution in $\alextwo$.
\begin{defn}\label{defnctsprimitiveintegral}
Let $f\in\alextwo$ with primitive $F\in\balextwo$.  We define its continuous primitive integral on 
interval $I=[a,b]\times [c,d]\subseteq\Rbar^2$ by
$\int_I f=\int_{a}^b\int_{c}^d f=F(a,c)+F(b,d)-F(a,d)-F(b,c)$.
\end{defn}
If $a=b$ or $c=d$ then the integral of $f$ over $I$ is zero.  This shows the integral over any
line parallel to the $x$ or $y$ axis is zero.  Hence, the integral over the boundary of an
interval always vanishes and there is no distinction between integrating over open or closed
intervals.  We also have the usual convention
that $\int_{b}^a\int_{c}^d f =-\int_{a}^b\int_{c}^d f=-\int_{a}^b\int_{d}^c f$.  
Note that $\int_{\Rbar^2}f=F(\infty,\infty)$ and $\int_{-\infty}^x\int_{-\infty}^y f=F(x,y)$ for all
$(x,y)\in\Rbar^2$.  As well, $\int_I(c_1f+c_2g)=c_1\int_I f+c_2\int_I g$.

The definition builds in the fundamental theorem of calculus.

\begin{prop}[Fundamental theorem of calculus]\label{propftc}
(a) Let $f\in\alextwo$ and define $\Phi(x,y)=\intxy f$.  Then $\Phi\in\balextwo$ and $\dalex\Phi=f$.
(b) Let $G\in C(\Rbar^2)$.  Then $\intxy\dalex G=G(-\infty,-\infty)+G(x,y)-G(-\infty,y)-G(x,-\infty)$.
\end{prop}
\begin{proof}
(a) See Proposition~\ref{propunique} (a).
(b) See Proposition~\ref{propunique} (c).
\end{proof}

The space $\balextwo$ is separable and hence $\alextwo$ is as well.
\begin{prop}\label{propseparable}
(a) Step functions are dense in $\balextwo$.
(b) Both $\balextwo$ and $\alextwo$ are separable.
(c) The real analytic functions are dense in $\balextwo$ and $\alextwo$.
(d) If $f\fn\R^2\to\R$ is a function in $L^1(\R^2)$ (with respect to Lebesgue measure),
or integrable in the sense of 
Henstock--Kurzweil or as a Denjoy integral then $f\in\alextwo$ and the integrals agree on
intervals in $\Rbar^2$.
\end{prop}
\begin{proof}
(a) For each $n\in\N$ we can make a partition of $\Rbar$ by $-\infty=p_0<p_1<\ldots<p_n=\infty$ and hence
of $\Rbar^2$ using $(p_i,p_j)$.  Let $P_{ij}=(p_{i-1},p_i]\times(p_{j-1},p_j]$.  Let $\sigma_{ij}\in\R$.
A step function is 
$$
\sigma(x,y)=\sum_{i=1}^n\sum_{j=1}^n\sigma_{ij}\chi_{P_{ij}}(x,y)
$$
with $\sigma(-\infty,y)=\sigma(x,-\infty)=0$ for $x,y\in\Rbar$.  If $p_i$ and $\sigma_{ij}$ are taken in $\Q$
then the collection of all such step functions is countable.    Since $\Rbar^2$ is compact, given $\epsilon>0$
and $F\in\balextwo$ there is a step function $\sigma$ with $\sigma_{1j}=\sigma_{i1}=0$ and 
$\norm{F-\sigma}_\infty<\epsilon$.

(b) The half-space Poisson kernel is $\Phi_z(x,y)=z(x^2+y^2+z^2)^{-3/2}/(2\pi)$ where $z>0$.  For example, see
\cite{axler}.  Note that $\intinf\intinf\Phi_z(x,y)\,dy\,dx=1$. 
For a step function $\sigma$ as above, define
\begin{eqnarray}
u_z(x,y) & = & \sigma\ast\Phi_z(x,y)=\intinf\intinf\sigma(x-\xi,y-\eta)\Phi_z(\xi,\eta)\,d\eta\,d\xi
\label{poisson}\\
 & = & \Phi_z\ast\sigma(x,y)=\intinf\intinf\sigma(\xi,\eta)\Phi_z(x-\xi,y-\eta)\,d\eta\,d\xi.
\label{poisson2}
\end{eqnarray}
In \eqref{poisson} we can have $(x,y)\in\Rbar^2$; since $\sigma(x,y)$ has limits with one of $x$ and $y$ fixed
in $\R$ and the other going to $\infty$ or $-\infty$ we define 
$$
u_z(\infty,y)= \sigma\ast\Phi_z(\infty,y)=\intinf\intinf\sigma(\infty,y-\eta)\Phi_z(\xi,\eta)\,d\eta\,d\xi.
$$
Similarly at other points  of the boundary of $\R^2$. With this convention, note that 
$$
u_z(\infty,\infty)=\sigma\ast\Phi_z(\infty,\infty)=\intinf\intinf\sigma(\infty,\infty)\Phi_z(\xi,\eta)\,d\eta\,d\xi = \sigma(\infty,\infty)
$$
since the Poisson kernel integrates to $1$.  Similarly at other corners of $\Rbar^2$.  We use \eqref{poisson2} only
for $(x,y)\in\R^2$.

If $(x,y)$ is in a compact set $K\subset\R^2$ then there is a constant $k$ (depending on $K$ and $z$) such
that $[(x-\xi)^2 + (y-\eta)^2 +z^2]^{-3/2}\leq k[\xi^2 + \eta^2 +z^2]^{-3/2}$ for all $(\xi,\eta)\in\R^2$.
By dominated convergence we can differentiate under the integral in \eqref{poisson2} at each $(x,y,z)\in\R^3$.  Hence,
$u_z\in C^\infty(\R^3)$.  (Harmonic functions are in fact real analytic.)  

Since $\sigma$ is bounded, dominated convergence allows us to take limits under the integral and this
shows $u_z$ is continuous on the boundary of $\R^n$ at points of continuity of $\sigma$.  To show
continuity at other points on the boundary note that $\sigma_z$ is the sum of a finite number of
 terms of type, say,
$$
\sigma_{mj}\int_{p_{m-1}}^\infty \int_{p_{j-1}}^{p_j}\Phi_z(\xi-x,\eta-y)\,d\eta\,d\xi
=\sigma_{mj}\int_{p_{m-1}-x}^\infty \int_{p_{j-1}-y}^{p_j-y}\Phi_z(\xi,\eta)\,d\eta\,d\xi
$$
which is clearly continuous on $\Rbar^2$.
Hence, $u_z\in\balextwo$.

Convolution with the Poisson kernel is known to approximate a continuous function uniformly on compact
sets in $\R^2$ as $z$ decreases to $0$.  
Our function $\sigma$ need not be continuous but is still approximated
in this sense.   If $(0,0)\in P_{ij}$ let $Q_0=\cup\{P_{\alpha\beta}\mid\alpha=i-1,i,i+1,\beta=j-1,j,j+1\}$.
If $i$ or $j$ is $1$ or $n$ then this union might contain fewer than nine rectangles.  
Since $F$ is continuous and $\norm{F-\sigma}_\infty<\epsilon$ we have $\abs{\sigma_{\alpha\beta}-\sigma_{\gamma\delta}}
\leq 2\epsilon$ if $\abs{\alpha-\gamma}\leq 1$ and $\abs{\beta-\delta}\leq 1$.
For $(x,y)\in\Rbar^2$ we have 
$$
u_z(x,y)-\sigma(x,y)  =  \sigma\ast\Phi_z(x,y)-\sigma(x,y)\intinf\intinf\Phi_z(\xi,\eta)\,d\eta\,d\xi
$$
so that
$\abs{u_z(x,y)-\sigma(x,y)}
\leq I_1+I_2$ where
\begin{eqnarray*}
I_1 & = & \int\int_{Q_0}\abs{\sigma(x-\xi,y-\eta)-\sigma(x,y)}\Phi_z(\xi,\eta)\,d\eta\,d\xi\\
 & \leq & 2\epsilon\int\int_{Q_0}\Phi_z(\xi,\eta)\,d\eta\,d\xi \leq 2\epsilon,\\
I_2 & = & \int\int_{\R^2\setminus Q_0}\abs{\sigma(x-\xi,y-\eta)-\sigma(x,y)}\Phi_z(\xi,\eta)\,d\eta\,d\xi\\
 & \leq & 2(\norm{F}_\infty+\epsilon)\int\int_{\R^2\setminus Q_0}\Phi_z(\xi,\eta)\,d\eta\,d\xi
\to 0 \text{ as } z\downarrow 0.
\end{eqnarray*}
The last line above follows with dominated convergence.  If we let $z$ decrease to $0$ through rational values
then we see $\balextwo$ is separable.  This also shows $\alextwo$ is separable.

(c) The proof of (b) shows the real analytic functions are dense in $\balextwo$ and hence dense in $\alextwo$.
(d)
Primitives of Lebesgue integrable, Henstock--Kurzweil integrable and Denjoy integrable functions are
continuous.  When an absolutely continuous function is differentiated pointwise the derivative agrees with
the distributional derivative.  The same applies to primitives of the other two integrals, which are discussed in
\cite{celidze}.  Hence, the continuous primitive integral includes the Lebesgue, Henstock--Kurzweil and 
Denjoy integrals in the sense that the integrals agree on intervals.
\end{proof}

It is known that $C(X)$ is separable exactly when $X$ is a compact metric space.  For example, 
\cite[p.~221]{kadison}. This then shows that $\balextwo$ is separable.  
However, our construction in the above proof lets us conclude real analytic
functions are dense in $\balextwo$.

If $f$ is a function in $L^1(\R^2)$ then $f\in\alextwo$ and $\norm{f}\leq\norm{f}_1$ with equality if
$f\geq 0$ almost everywhere.  The norms are not equivalent on $L^1(\R^2)$.  For example, for $n\in\N$ let 
$f_n(x,y)=\sin(nx)\chi_{(0,2\pi)}(x)\chi_{(0,1)}(y)$.  Then $\norm{f_n}=\int_0^{\pi/n}\sin(nx)\,dx=2/n$
and $\norm{f_n}_1=\int_0^{2\pi}\abs{\sin(nx)}\,dx=4$. Hence, there can be no inequality
$c_1\norm{f}\leq \norm{f}_1\leq c_2\norm{f}$ for some constants $c_1$, $c_2$ and all $f\in L^1(\R^2)$.

\begin{prop}\label{propdense}
(a) $L^1(\R^2)$ is dense in $\alextwo$.
(b) The test functions are dense in $\alextwo$.
\end{prop}
\begin{proof}
(a) The construction in Proposition~\ref{propseparable} shows $L^1(\R^2)$ is dense in $\alextwo$ since
we can differentiate $u_z$ under the integral sign.  Integration then shows
$\norm{\dalex u_z}_1\leq 4\sum_{i,j}\abs{\sigma_{ij}}$.
(b) If $f\in\alextwo$ and $\epsilon>0$ there is $g\in L^1(\R^2)$ such that $\norm{f-g}<\epsilon$.
If $\phi$ is a test function then $\norm{f-\phi}\leq \norm{f-g}+\norm{g-\phi}_1$ and test functions
are known to be dense in $\Lone$.  For example, \cite[Proposition~8.17]{folland}.
\end{proof}

This then gives an alternative way to define the integral.  We have $\alextwo$ is the completion of
$\Lone$ in the Alexiewicz norm.  If $\{f_n\}\subset L^1(\R^2)$ is a Cauchy sequence in the 
Alexiewicz norm then it converges to a distribution $f\in\alextwo$.  Let $F_n, F\in\balextwo$ be the
respective primitives of $f_n$ and $f$.  Then $\dalex(F_n-F)=\dalex F_n-\dalex F=f_n-f$ so
$\norm{F_n-F}_\infty\to 0$.  This gives
\begin{eqnarray*}
\intabcd f & = & F(a,c)+F(b,d)-F(a,d)-F(b,c)\\
 & = & \lim_{n\to\infty}\left[F_n(a,c)+F_n(b,d)-F_n(a,d)-F_n(b,c)\right]\\
& = & \lim_{n\to\infty}\intabcd f_n
\end{eqnarray*}
and defines the integral of $f\in\alextwo$ using only Lebesgue integrals of functions in $L^1(\R^2)$.

\section{Examples}\label{sectionexamples}
If $f$ and $\tilde{f}$ are functions in $\alextwo$ such that $f$ and $\tilde{f}$ agree except on a set
of Lebesgue measure zero  then they have the same primitive in $\balextwo$ and hence the same integral
on all subintervals on $\Rbar^2$.  Of course, this pointwise comparison is not possible for distributions
in $\alextwo$ that do not happen to be functions.

Functions that have a conditionally convergent Henstock--Kurzweil or improper Riemann integral are 
in $\alextwo\setminus\Lone$.
For example, we can define $f(x,y)=\sin(x)\sin(y)/(xy)$ with
$f(x,y)=0$ if $x=0$ or $y=0$.  For another example take
$
G(x,y)=
x^2y^2\sin(x^{-4})\sin(y^{-4})$ with $G(x,y)=0$ if $x=0$ or $y=0$.  Then $G\in\balextwo$ so 
$\dalex G\in\alextwo\setminus L^1_{loc}(\R^2)$.

The above examples are products of a function of $x$ and a function of $y$.  More generally,
note that $\balex$ and $\balextwo$ are closed under pointwise products.  See Section~\ref{sectionbanachalgebra}.
Hence, if $F,G\in\balextwo$ so
is $FG$ and $\dalex(FG)\in\alextwo$.  
Similarly, the function $(x,y)\mapsto F(x,y)G(x)\in\balextwo$ if now $G\in C(\Rbar)$  or if
$G\in C((-\infty,\infty])$ and is bounded. 
In general we cannot apply a differentiation product rule
except when $G$ is of bounded variation.  See Section~\ref{sectionparts}.  

If $f,g\in\alex$ then
define $h\in\alextwo$ by $h(x,y)=f(x)g(y)$.  This is in fact a tensor product but we will not employ
any special notation.  We can take $F,G\in\balex$ to be functions of Weierstrass type that are
continuous but pointwise differentiable nowhere.  Then the distributional derivative is $\dalex(FG)=f'g'$.
Neither the Lebesgue nor Henstock--Kurzweil integral of $f'g'$ exists on any interval but the continuous
primitive integral is $\int_a^b\int_c^df'g'=[F(b)-F(a)][G(d)-G(c)]$ for all $[a,b]\times[c,d]\subseteq\Rbar^2$.
If we take $F,G\in\balex$ to be singular, i.e., continuous, not constant, with pointwise derivative equal
to $0$ almost everywhere, then the Lebesgue integral exists and gives $\int_a^b\int_c^df'g'=0$ while the
continuous primitive integral is again $\int_a^b\int_c^df'g'=[F(b)-F(a)][G(d)-G(c)]$.

An example of $F\in\balextwo$ that is not a product of functions in $\balex$ is $F(x,y)=\exp(-\sqrt{x^2+y^2})$.
Proposition~\ref{propboundaryvalues} also gives a procedure for constructing such primitives.

In Section~\ref{sectionchange} we discuss change of variables.  It is worth noting here that $C(\Rbar^2)$, $\balextwo$,
$\alextwo$ and $\dalex$ are not invariant under rotations.  For example, if $F(x,y)=x/(1+\abs{x})$ then
$F\in C(\Rbar^2)$ and $\dalex F=0$.  Rotate to get $G(x,y)=(x+y)/(1+\abs{x+y})$.  For $(x,y)\in\R^2$
we have $G(\infty,y)=1$, $G(-\infty,y)=-1$, $G(x,\infty)=1$, $G(x, -\infty)=-1$.  Hence, $G\not\in C(\Rbar^2)$.
And, $\dalex G\not=0$.

The topology of $\Rbar^2$ depends on the Cartesian coordinate system.
In $\Rbar^2$ we employ a four-point compactification of $\R^2$.  Stereographic projection uses a one-point
compactification so that a function continuous in the polar coordinates extended plane must 
have $\lim_{r\to\infty}F(r,\theta)$ equal
to a constant, independent of angle $\theta$.  If $F$ is continuous in this sense then  $F\in C(\Rbar^2)$.
The converse is not true; for example, $F(x,y)=\arctan(x)\arctan(y)$.

Also, in polar coordinates we could use a compactification of $\R^2$ with a continuum of points at infinity.
In this sense, a function $(r,\theta)\mapsto F(r,\theta)$ is continuous on this extended real plane
at $r=\infty$, $\alpha\in[-\pi,\pi]$, if
for each $\epsilon>0$ there is $\delta>0$ such that if $r>1/\delta$ and $\abs{\theta-\alpha}<\delta$ then
$\abs{F(r,\theta)-F(\infty,\alpha)}<\epsilon$.  This topology is neither coarser nor finer than that for
$\Rbar^2$.  For example,
the function
$F(x,y)=x/(1+\abs{x})$ is continuous on $\Rbar^2$.  Introduce polar coordinates by
defining $G(r,\theta)=F(r\cos\theta,r\sin\theta)=r\cos\theta/(1+r\abs{\cos\theta})$.  Then
$$
G(\infty,\theta)=\left\{\begin{array}{cl}
\frac{\cos\theta}{\abs{\cos\theta}}, & \abs{\theta}\not=\frac{\pi}{2}\\
0, & \abs{\theta}=\frac{\pi}{2}.
\end{array}
\right.
$$
Hence, $G$ is not continuous on the extended polar coordinates plane.
And, the function $G(r,\theta)=\sqrt{r}\sin(\theta)/(1+\sqrt{r})$ is continuous in polar
coordinates (continuum of points at infinity, not one-point compactification).  In Cartesian coordinates, $G$ 
becomes $F(x,y)=y/((x^2+y^2)^{1/4}+\sqrt{x^2+y^2})$.  And, $F(\pm\infty,y)=0$ for $y\in\R$,
$F(x,\pm\infty)=\pm 1$ for $x\in\R$ so $F\not\in C(\Rbar^2)$.

For any prescribed continuous function on $\partial\R^2$ there is a function in $C(\Rbar^2)$ with these
boundary values.

\begin{prop}\label{propboundaryvalues}
Suppose there are functions $\Theta_i\in C(\Rbar)$ such that $\Theta_1(\infty)=\Theta_2(-\infty)$,
$\Theta_2(\infty)=\Theta_3(\infty)$, $\Theta_3(-\infty)=\Theta_4(\infty)$, $\Theta_1(-\infty)=\Theta_4(-\infty)$.
Then there is a function $F\in C(\Rbar^2)$ such
that $F(-\infty,y)=\Theta_1(y)$, $F(x,\infty)=\Theta_2(x)$, $F(\infty,y)=\Theta_3(y)$, $F(x,-\infty)=\Theta_4(x)$.
\end{prop}
\begin{proof}
First consider $\Theta_1=\Theta_4=0$.
Define
$$
F(x,y)=\left\{\begin{array}{cl}
\frac{\Theta_2(x)\Theta_3(y)}{\Theta_2(\infty)}, & \text{ if }
\Theta_2(\infty)\not=0\\
\frac{\Theta_2(x)}{\pi}\left[\frac{\pi}{2}+\arctan(y)\right]
+\frac{\Theta_3(y)}{\pi}\left[\frac{\pi}{2}+\arctan(x)\right], & \text{ if }
\Theta_2(\infty)=0.
\end{array}
\right.
$$
Now consider $\Theta_2=\Theta_3=0$.  Add the resulting functions.
\end{proof}
Note that the first part of the proof gives a function in $\balextwo$.  Differentiating then generates
a wealth of examples of distributions in $\alextwo$.

\section{Hardy--Krause bounded variation}\label{sectionhkbv}
If $g\fn\Rbar\to\R$ then the variation of $g$ is $V\!g=\sup\sum_{i=1}^N\abs{g(x_i)-g(x_{i-1})}$
where the supremum is taken over all partitions $-\infty\leq x_0<x_1<\ldots<x_N\leq\infty$.
The set of functions of bounded variation, denoted $\bv(\Rbar)$, is a Banach space under the norm $\norm{g}_{BV}=
\norm{g}_\infty+V\!g$.  If $f\in\alex$ with primitive $F\in \balex$ and $g\in\bv(\Rbar)$ then there is the integration by parts
formula $\intinf fg= F(\infty)g(\infty)-\intinf F\,dg$.  The last integral is a Henstock--Stieltjes integral.
See \cite{mcleod}.  If we take $\lambda_1, \lambda_2\in[0,1]$ such that $\lambda_1+\lambda_2=1$ (a convex combination)
and require $g\in\bv$ to satisfy $g(x)=\lambda_1g(x-)+\lambda_2g(x+)$ for each $x\in\R$ and $g(\infty)=\lim_{x\to\infty}
g(x)$, $g(-\infty)=\lim_{x\to-\infty}g(x)$ then $g$ is said to be of {\it normalised bounded variation}.  
For example,
taking $\lambda_1=0$ and $\lambda_2=1$ makes $g$ right continuous on $\R$. Sometimes different conditions
are imposed at $\pm\infty$.  A function of bounded variation need
only be changed on a countable collection of points to make it of normalised bounded variation.  
Note that the $\bv$ norm of a function is the same for any normalisation.  A normalisation can then be fixed and
the resulting space labeled $\nbv$.

Two intervals in $\Rbar^2$ are {\it nonoverlapping} if their intersection is of Lebesgue (planar)
measure zero.  A {\it division} of $\Rbar^2$ is a finite collection of nonoverlapping intervals
whose union is $\Rbar^2$.  If $g\fn\Rbar^2\to\R$ then its Vitali variation is $V_{12}g=\sup_D\sum_i\abs{
g(a_{i},c_{i})+g(b_i,d_i)-g(a_i,d_i)-g(b_i,c_i)}$ where the supremum is taken over all divisions
$D$ of $\Rbar^2$ and interval $I_i=[a_i,b_i]\times[c_i,d_i]$ is an interval in $D$.  If we fix one
variable and find the one-variable variation as a function of the remaining variable we write
$V_1g(\cdot,y_0)$ or $V_2g(x_0, \cdot)$ according as the second variable has been fixed as $y_0$
or the first variable fixed as $x_0$.  The space of Hardy--Krause bounded variation is defined as
follows.
\begin{defn}[Hardy--Krause variation]
Let $g\fn\Rbar^2\to\R$.  Suppose $V_{12}g$ is finite and for some $x_0$ and $y_0$ in $\Rbar$ and both
$V_1g(\cdot,y_0)$ and $V_2g(x_0, \cdot)$ are finite.  Then $g$ is said to be of finite
Hardy--Krause variation.  The set of all such functions is denoted $\hkbv(\Rbar^2)$.  We write 
$\norm{g}_{bv}=\norm{g}_\infty+\norm{V_1g}_\infty+\norm{V_2g}_\infty
+ V_{12}g$. 
\end{defn}
Basic results about functions of finite Hardy--Krause variation are proved in \cite{clarksonadams} and
\cite{appell}.  Our definition is slightly different but the same results hold.  In particular,
functions in $\hkbv(\Rbar^2)$ are bounded and
if $V_1g(\cdot,y_0)$ and $V_2g(x_0, \cdot)$ are finite for some $x_0$ and $y_0$ then $\norm{V_1g}_\infty$
and $\norm{V_2g}_\infty$ are finite and
$\hkbv(\Rbar^2)$ is a Banach space.

There are many types of variation for functions of two or more variables but Hardy--Krause variation
is the most appropriate for nonabsolute integration.  See \cite{clarksonadams}, \cite{adamsclarkson},
\cite{adamsclarksoncorrection} and
\cite{hildebrandt}.

\begin{example}\label{examplegbv}
For $(x,y)\in\R^2$ let $g=\chi_{[-\infty,x)\times[-\infty,y)}$.  Then $\norm{g}_\infty=1$.  And,
$$
V_1g(\cdot,t)  =  \left\{\begin{array}{cl}
0, & t\geq y\\
1, & t<y,
\end{array}
\right.
$$
$$
V_2g(s,\cdot)  =  \left\{\begin{array}{cl}
1, & s<x\\
0, & s\geq x,
\end{array}
\right.
$$
so that $\norm{V_1g}_\infty=\norm{V_2g}_\infty=1$.  Note that $V_{12}g=1$ since there is exactly
one interval in each division of $\Rbar^2$ with exactly one corner in $[-\infty,x)\times[-\infty,y)$.
Therefore $\norm{g}_{bv}=4$.

Similarly, if 
$$
g(x,y)=\left\{\begin{array}{cl}
1, & x\geq 0\\
0, & x<0
\end{array}
\right.
$$
then $\norm{g}_\infty=1$, $\norm{V_1g}_\infty=1$, $\norm{V_2g}_\infty=0$, $V_{12}g=0$ so that $\norm{g}_{bv}=2$.

If $I$ is a finite interval in
$\R^2$ and $g=\chi_I$ then we have $\norm{g}_\infty=1$, $\norm{V_1g}_\infty=\norm{V_2g}_\infty=2$, $V_{12}g=4$ 
so that $\norm{g}_{bv}=9$.
\end{example}

\begin{example}
The function
$$
g(x,y)=\left\{\begin{array}{cl}
1, & y>x\\
0, & y\leq x.
\end{array}
\right.
$$
is not of bounded Hardy--Krause variation.  Only intervals with one corner on the line $y=x$ contribute
to $V_{12}$.  But there can be a countable number of these.  For a similar example see \cite{ostaszewski2}.

It can be shown that $\chi_I$ is of bounded Hardy--Krause variation if and only if $I$ is
a finite union of intervals (in the fixed Cartesian coordinate system).
\end{example}

\section{Integration by parts}\label{sectionparts}
Note the classical formula, valid for $F,g\in C^2(\R^2)$ and all $a,b,c,d\in\R$,
\begin{align*}
&\intabcd\dalex(Fg)(x,y)\,dy\,dx\\
&=  F(a,c)g(a,c)+F(b,d)g(b,d)-F(a,d)g(a,d)-F(b,c)g(b,c)\\
&=\intabcd F_{12}(x,y)g(x,y)\,dy\,dx+\int_a^b[F(x,d)g_1(x,d)-F(x,c)g_1(x,c)]\,dx\\
&+\int_c^d[F(b,y)g_2(b,y)-F(a,y)g_2(a,y)]\,dy-\intabcd F(x,y)g_{12}(x,y)\,dy\,dx.
\end{align*} 
This gives us the form the integration by parts formula should have in $\alextwo$.
It is essentially the same as the formula for Henstock--Kurzweil integrals
\cite[Example~6.5.11]{lee}.  See also \cite{young}.
\begin{defn}[Integration by parts]\label{defnparts}
Let $f\in\alextwo$ with primitive $F\in\balextwo$.  Let $g\in\hkbv(\Rbar^2)$.  Let 
$[a,b]\times[c,d]\subseteq\Rbar^2$.
Define 
\begin{eqnarray*}
\intabcd fg & = & F(a,c)g(a,c)+F(b,d)g(b,d)-F(a,d)g(a,d)-F(b,c)g(b,c)\\
 & - & \int_a^b F(x,d)\,d_1g(x,d)+\int_a^bF(x,c)\,d_1g(x,c)\\
 & - & \int_c^dF(b,y)\,d_2g(b,y)+\int_c^dF(a,y)\,d_2g(a,y)\\
 & + & \intabcd F(x,y)\,d_{12}g(x,y).
\end{eqnarray*}
\end{defn}
Subscripts indicate a Henstock--Stieltjes integral with respect to the relevant variable.  This is defined
as follows \cite{bongiornopap}.  A {\it tagged division} of $\Rbar^2$ is a division for which each interval in the division
has an associated {\it tag}, which is a point in the interval.  A {\it gauge} is a mapping $\gamma$ from $\Rbar^2$ to
the open sets in $\Rbar^2$ such that $\gamma(x,y)$ is an open set containing point $(x,y)$.  An
interval-point pair in a tagged division is $\gamma$-fine if $I\subset\gamma(x,y)$ where $(x,y)$ is the
tag associated with interval $I$.  It is possible to choose $\gamma$ such that if $(x,y)\in\R^2$ then
$\gamma(x,y)\subset\R^2$.  This means that if an interval in a $\gamma$-fine tagged division intersects
the boundary of $\R^2$ then its tag must be on the boundary of $\R^2$.  (We always assume this of $\gamma$.)
Existence of $\gamma$-fine
tagged divisions is proven in  \cite{lee},
\cite{mcleod} and \cite{swartz}, and is a consequence of the compactness of $\Rbar^2$.  
If $F,g\fn\Rbar^2\to\R$ then the Henstock--Stieltjes integral 
$\intinf\intinf F(x,y)\,d_{12}g(x,y)$ exists with value $A\in\R$ if for every $\epsilon>0$ there
is a gauge $\gamma$ such that for each $\gamma$-fine tagged division $\{[a_i,b_i]\times[c_i,d_i],(x_i,y_i)\}_{i=1}^N$
we have
$$
\left|\sum_{i=1}^N F(x_i,y_i)[g(a_{i},c_{i})+g(b_i,d_i)-g(a_i,d_i)-g(b_i,c_i)]-A\right|<\epsilon.
$$

There are various other Stieltjes type integrals, including Riemann-Stieltjes, but they are all equivalent
when $F$ is continuous.  See \cite{hildebrandt}, \cite{mcleod} and \cite{hanungtvrdy}.

Note that according to Definition~\ref{defnparts}
\begin{eqnarray*}
\int_{-\infty}^x\int_{-\infty}^y fg & = & F(x,y)g(x,y)-\int_{-\infty}^x F(s,y)\,d_1g(s,y)\\
 & & \qquad-\int_{-\infty}^y F(x,t)\,d_2g(x,t)
+\int_{-\infty}^x\int_{-\infty}^y F(s,t)\,d_{12}g(s,t).
\end{eqnarray*}

\begin{remark}\label{remarkmultipliersmodule}
Every distribution, $T$, can be multiplied by every $C^\infty$ function, $\psi$, using
$\langle T\psi,\phi\rangle=\langle T, \phi\psi\rangle$ for test function $\phi$.  The
pointwise product $\phi\psi$ is again a test function.  And, Definition~\ref{defnparts}
now defines
the product $fg$.  Since $fg$ is integrable for each $f\in\alextwo$ we say $g\in\hkbv(\Rbar^2)$ is
a {\it multiplier} for the continuous primitive integral.  

The integration by parts formula then induces the
multiplication $\alextwo\times\hkbv\to\alextwo$ and $\alextwo$ is then a Banach $\hkbv$-module.  See
\cite{dales} for the definition.  Theorems similar to those in \cite[\S7]{talvilaregulated}
can then be proved.
\end{remark}

We can justify the above definition with the following observation.
\begin{prop}\label{propparts}
Suppose $F\in C(\Rbar^2)$, $f=\dalex F$, $g\in\hkbv(\Rbar^2)$.  For $(x,y)\in\Rbar^2$ define
\begin{eqnarray*}
\Phi(x,y) & = & F(x,y)g(x,y)-\int_{-\infty}^x F(s,y)\,d_1g(s,y)-\int_{-\infty}^y F(x,t)\,d_2g(x,t)\\
 &  & \qquad + \int_{-\infty}^x\int_{-\infty}^y F(s,t)\,d_{12}g(s,t).
\end{eqnarray*}
Then $\Phi\in C(\Rbar^2)$.  If $F\in\balextwo$ then $\Phi\in\balextwo$ and $\dalex\Phi\in\alextwo$.
\end{prop}
\begin{proof}
To prove continuity at $(x,y)\in\R^2$ let $(\xi,\eta)\in\R^2$. It suffices to consider $\xi\leq x$
and $\eta\leq y$.  Then
\begin{align}
&\Phi(x,y)-\Phi(\xi,\eta)  =  F(x,y)g(x,y)-F(\xi,\eta)g(\xi,\eta)\label{phi1}\\
&\qquad - \int_{-\infty}^x F(s,y)\,d_1g(s,y)+ \int_{-\infty}^{\xi} F(s,\eta)\,d_1g(s,\eta)\label{phi2}\\
&\qquad - \int_{-\infty}^y F(x,t)\,d_2g(x,t)+ \int_{-\infty}^\eta F(\xi,t)\,d_2g(\xi,t)\label{phi3}\\
&\qquad+ \int_{-\infty}^x\int_{-\infty}^y F(s,t)\,d_{12}g(s,t)-
\int_{-\infty}^\xi\int_{-\infty}^\eta F(s,t)\,d_{12}g(s,t).\label{phi4}
\end{align}
The right side of \eqref{phi1} is written 
\begin{equation}
[F(x,y)-F(\xi,\eta)]g(\xi,\eta)+F(x,y)[g(x,y)-g(\xi,\eta)].\label{P1}
\end{equation}
Line \eqref{phi2} is written
$$
- \int_{-\infty}^\xi[F(s,y)-F(s,\eta)]\,d_1g(s,\eta)
+\int_{-\infty}^\xi F(s,y)\,d_1g(s,\eta)
 - \int_{-\infty}^x F(s,y)\,d_1g(s,y).
$$
Line \eqref{phi3} is written
$$
-\int_{-\infty}^\eta[F(x,t)- F(\xi,t)]\,d_2g(\xi,t) +\int_{-\infty}^\eta
F(x,t)\,d_2g(\xi,t)
- \int_{-\infty}^y F(x,t)\,d_2g(x,t).
$$
Line \eqref{phi4} is written
\begin{align*}
&\int_{-\infty}^\xi\int_{\eta}^y[F(s,t)-F(s,y)]\,d_{12}g(s,t)
+\int_{-\infty}^\xi F(s,y)\,[d_{1}g(s,y) -d_1g(s,\eta)]\\
&+\int_{\xi}^x\int_{-\infty}^\eta[F(s,t)-F(x,t)]\,d_{12}g(s,t)
+\int_{-\infty}^\eta F(x,t)\,[d_{2}g(x,t) -d_2g(\xi,t)]\\
&+\int_{\xi}^x\int_{\eta}^y[F(s,t)-F(x,y)]\,d_{12}g(s,t)
+F(x,y)\int_{\xi}^x[d_1g(s,y)-d_1g(s,\eta)].
\end{align*}
Combining the above four terms gives
\begin{align*}
&\abs{\Phi(x,y)-\Phi(\xi,\eta)}\leq \abs{F(x,y)-F(\xi,\eta)}\abs{g(\xi,\eta)}\\
&+\left|\int_{-\infty}^\xi[F(s,y)-F(s,\eta)]\,d_1g(s,\eta)\right|
+\left|\int_{\xi}^x[F(x,y)-F(s,y)]\,d_1g(s,y)\right|\\
&+\left|\int_{-\infty}^\eta[F(x,t)-F(\xi,t)]\,d_2g(\xi,t)\right|
+\left|\int_{\eta}^y[F(x,y)-F(x,t)]\,d_2g(x,t)\right|\\
&+\left|\int_{-\infty}^\xi\int_\eta^y[F(s,t)-F(s,y)]\,d_{12}g(s,t)\right|
+\left|\int_{\xi}^x\int_{-\infty}^\eta [F(s,t)-F(x,t)]\,d_{12}g(s,t)\right|\\
&+\left|\int_{\xi}^x\int_\eta^y [F(s,t)-F(x,y)]\,d_{12}g(s,t)\right|.
\end{align*}
The integrals in the above line with respect to $d_1g$ are bounded by $2\norm{F}_\infty\norm{V_1g}_\infty$;
those with respect to $d_2g$ are bounded by $2\norm{F}_\infty\norm{V_2g}_\infty$;
those with respect to $d_{12}g$ are bounded by $2\norm{F}_\infty V_{12}g$.
By dominated convergence and the continuity of $F$ it then follows that they all tend to $0$ as
$(\xi,\eta)\to(x,y)$.  This shows $\Phi$ is continuous on $\R^2$.
Minor changes show continuity on $\Rbar^2$.  It follows from the definition of $\Phi$ that if
$F\in\balextwo$ then $\Phi\in\balextwo$.
\end{proof}
\begin{remark}\label{remarkparts}
Note that Definitions \ref{defnctsprimitiveintegral} and \ref{defnparts} agree in the case when $g$ is the 
characteristic function of an interval.  Suppose $g=\chi_I$ where $I=[a,b]\times[c,d]$ is a compact
interval in $\R^2$.  Since $F(x,y)$ vanishes when $x=-\infty$ or $y=-\infty$,
Definition \ref{defnparts} gives
\begin{eqnarray*}
\intinf\intinf fg & = & F(\infty,\infty)g(\infty,\infty)-\int_{-\infty}^\infty F(x,\infty)\,d_1g(x,\infty)\\
 & &\qquad-\int_{-\infty}^\infty F(\infty,y)\,d_2g(\infty,y)
+\int_{-\infty}^\infty\int_{-\infty}^\infty F(x,y)\,d_{12}g(x,y)\\
  & = & \int_{-\infty}^\infty\int_{-\infty}^\infty F(x,y)\,d_{12}g(x,y),
\end{eqnarray*}
last line following since $g$ is constant in a neighbourhood of $\partial\R^2$.

Now use Definition~\ref{defnctsprimitiveintegral}.
Suppose $[s,t]\times[u,v]$ is an interval in a tagged division of $\Rbar^2$.  Let
$\Delta g=g(s,u)+g(t,v)-g(s,v)-g(t,u)$.  A Riemann sum consists of terms $F(z_1,z_2)\Delta g$
for some tag $(z_1,z_2)\in[s,t]\times[u,v]$.  Consider the point $(a,c)$.  
For any gauge $\gamma$, a $\gamma$-fine tagged division can be chosen so that
$(a,c)$ is in the interior of exactly one interval and $(a,c)$ is the tag for this interval.
For this interval $\Delta g=1$.  Similarly with the points $(b,d)$, $(a,d)$, $(b,c)$.  And 
$\Delta g$ vanishes for all but four intervals in the Riemann sum.   This shows
$$
\int_{-\infty}^\infty\int_{-\infty}^\infty F(x,y)\,d_{12}g(x,y)=F(a,c)+F(b,d)-F(a,d)-F(b,c).
$$
Similarly, if $I\subseteq\Rbar^2$.
\end{remark}

By Proposition~\ref{propdense}
every distribution in $\alextwo$ is the limit in the Alexiewicz norm of a sequence of functions in $L^1(\R^2)$.
We can show how this also holds for $\dalex\Phi$ from Proposition~\ref{propparts}.
\begin{prop}\label{propPhin}
Let $f\in\alextwo$.  Let $\{f_n\}\subset L^1(\R^2)$ such that $\norm{f-f_n}\to 0$.
Let $F$ and $F_n$ be the respective primitives in $\balextwo$.
With $\Phi$ as in Proposition~\ref{propparts} and $\Phi_n$ similarly for $F_n$ we have
$\norm{\dalex\Phi_n-\dalex\Phi}\to 0$ as $n\to\infty$.
\end{prop}
\begin{proof}
We have the estimate
$$
\abs{\Phi_n(x,y)-\Phi(x,y)}  \leq  \norm{F_n-F}_\infty\left(\norm{g}_\infty+\norm{V_1g}_\infty+
\norm{V_2g}_\infty+V_{12}g\right)
$$
from which the result follows.
\end{proof}

If $\phi\in\Dtwo$ then $\phi\in\hkbv(\Rbar^2)$ so integration by parts gives another interpretation
of the action of $f\in\alextwo$ as a distribution.  Let $F\in\balextwo$ be the primitive of $f$.  We have
\begin{eqnarray*}
\langle f,\phi\rangle & = & \langle \dalex F,\phi\rangle=\langle F,\dalex\phi\rangle=
F(\infty,\infty)\phi(\infty,\infty)-\intinf F(s,\infty)\,d_1\phi(s,\infty)\\
 & & \qquad -\intinf F(\infty,t)\,d_2(\infty,t)
+\intinf\intinf f\phi\\
 & = & \intinf\intinf f\phi.
\end{eqnarray*}

If $F$ is a continuous function on $\Rbar$ and $g$ is of bounded variation the one-dimensional
formula is well-known:
$$
\intinf g\,dF=F(\infty)g(\infty)-F(-\infty)g(-\infty)-\intinf F\,dg.
$$
For example, see \cite{mcleod}.  There is an analogue in $\alextwo$.
\begin{prop}\label{propgdF}
Let $F\in\balextwo$ and $f=\dalex F$.  Let $g\in\hkbv(\Rbar^2)$.  
Then for $[a,b]\times[c,d]\subseteq\Rbar^2$
\begin{align*}
&\intabcd\left[F(a,c)+F(x,y)-F(a,y)-F(x,c)\right]d_{12}g(x,y)\\
&\qquad=\intabcd\left[g(x,y)+g(b,d)-g(x,d)-g(b,y)\right]d_{12}F(x,y).
\end{align*}
In particular,
\begin{eqnarray}
\intinf\intinf F\,d_{12}g
 & = & \intinf\intinf g\,d_{12}F+g(\infty,\infty)F(\infty,\infty)-\intinf g(x,\infty)\,d_1F(x,\infty)\notag\\
 & & \qquad-\intinf g(\infty,y)\,d_2F(\infty,y)\label{hildebrandtA1}
\end{eqnarray}
so that if $g(x,y)$ vanishes when $x$ or $y$ is $\infty$ then
\begin{equation}
\intinf\intinf fg=\intinf\intinf F\,d_{12}g=\intinf\intinf g\,d_{12}F.\label{hildebrandtA2}
\end{equation}
If $g(x,y)$ vanishes when $x$ or $y$ is $-\infty$ then
\begin{eqnarray}
\intinf\intinf fg & = & F(\infty,\infty)g(\infty,\infty)-\int_{-\infty}^\infty F(x,\infty)\,d_1g(x,\infty)\\
 & &\qquad-\int_{-\infty}^\infty F(\infty,y)\,d_2g(\infty,y)
+\int_{-\infty}^\infty\int_{-\infty}^\infty F\,d_{12}g\notag\\
  & = & \intinf\intinf g\,d_{12}F.\label{hildebrandtB2}
\end{eqnarray}

\end{prop}
\begin{proof}
The first line is from Theorem~8.8, page~127
in \cite{hildebrandt}, which is easily extended from a compact interval to $\Rbar^2$.
The author considers various Stieltjes integrals but these are all equal under the
hypotheses of our theorem.  This first line can be written
\begin{align*}
&F(a,c)[g(a,c)+g(b,d)-g(a,d)-g(b,c)]+\intabcd F\,d_{12}g\\
&\qquad-\int_c^d F(a,y)\left[d_2g(b,y)-d_2g(a,y)\right]
-\int_a^b F(x,c)\left[d_1g(x,d)-d_1g(x,c)\right]\\
&=\intabcd g\,d_{12}F+g(b,d)\left[F(a,c)+F(b,d)-F(a,d)-F(b,c)\right]\\
&\qquad-\int_a^b g(x,d)\left[d_1F(x,d)-d_1F(x,c)\right]
-\int_c^d g(b,y)\left[d_2F(b,y)-d_2F(a,y)\right].
\end{align*}
Taking limits gives \eqref{hildebrandtA1} and \eqref{hildebrandtA2}.  Interchanging
$F$ and $g$ in the first line and repeating these steps gives the final equations.
\end{proof}

\begin{prop}[First and second mean value theorem for integrals]
Suppose $g\in\hkbv(\Rbar^2)$ such that $g(a,c)+g(b,d)-g(a,d)-g(b,c)\geq 0$
for all $[a,b]\times[c,d]\subseteq\Rbar^2$.
(a) Suppose $F\in C(\Rbar^2)$.
Then there exists $(\xi,\eta)\in\Rbar^2$ such that
$$
\intinf\intinf F\,d_{12}g=F(\xi,\eta)\left[g(-\infty,-\infty)+g(\infty,\infty)-g(-\infty,\infty)-
g(\infty,-\infty)\right].
$$
(b) Let $f\in\alextwo$ and let $F\in\balextwo$ be its primitive.  
Suppose $g(x,y)$ vanishes when $x$ or $y$ is infinity.
Then there is $(\xi,\eta)\in\Rbar^2$ such that
$$
\intinf fg=F(\xi,\eta)\left[g(-\infty,-\infty)+g(\infty,\infty)-g(-\infty,\infty)-
g(\infty,-\infty)\right].
$$
\end{prop}
\begin{proof}
(a) Let $\Delta =g(-\infty,-\infty)+g(\infty,\infty)-g(-\infty,\infty)-
g(\infty,-\infty)$. The function $\Psi(x,y)=F(x,y)\intinf\intinf d_{12}g$ is continuous.  Both $\Psi$ and
$\intinf\intinf F\,d_{12}g$ have $(\min_{\Rbar^2}F)\Delta$ and $(\max_{\Rbar^2}F)\Delta$ as respective
lower and upper bounds.  Use of the intermediate value theorem completes the proof.
(b) Use part (a) and \eqref{hildebrandtA2}.
\end{proof}
Versions of the first mean value theorem, part (a), are known for Henstock--Stieltjes and 
Lebesgue integrals.  See page~209 in \cite{mcleod} and Problem~6, page~190, in \cite{fleming}.
For the second mean value theorem for one-dimensional Henstock--Kurzweil integrals see \S1.10 in
\cite{celidze} and page~211 in \cite{mcleod}.
See Theorems~6.4.2 and $6.5.13$ in \cite{lee} for $n$-dimensional Henstock--Kurzweil integrals.
See also \cite{young}.

The space $\alextwo$ is invariant under translations, as is the Alexiewicz norm.  We also have continuity
of translations.  If $f\in\Dptwo$ and $(s,t)\in\R^2$ then the translation is defined by
$\langle \tau_{(s,t)}f,\phi\rangle=\langle f,\tau_{(-s,-t)}\phi\rangle$ where $\tau_{(s,t)}\phi(x,y)
=\phi(x-s,y-t)$ for $\phi\in\Dtwo$.
\begin{prop}\label{propalexinvariance}
(a) If $f\in\Dptwo$ then $f\in\alextwo$ if and only if $\tau_{(s,t)}f\in\alextwo$ for all $(s,t)\in\R^2$.
(b) Let $f\in\alextwo$.  Then $\norm{f}=\norm{\tau_{(s,t)}f}$ for all $(s,t)\in\R^2$.
(c) Let $f\in\alextwo$.  Then $\lim_{(s,t)\to(0,0)}\norm{f-\tau_{(s,t)}f}=0$.
\end{prop}
The proofs are based on the corresponding properties in $\balextwo$.  See also \cite[Theorem~28]{talviladenjoy}.

\section{H\"older inequality and dual space}\label{sectionholder}
The integration by parts formula, Definition~\ref{defnparts}, leads to a version of the H\"older inequality.  It is known that the
dual space of the Henstock--Kurzweil integrable functions is $\hkbv(\Rbar^2)$ (\cite[\S6.6]{lee}).
The H\"older inequality and density of
the Henstock--Kurzweil integrable functions in $\alextwo$ then show that the dual space of
$\alextwo$ is also $\hkbv(\Rbar^2)$.

\begin{prop}[H\"older inequality]\label{propholder}
Let $f\in\alextwo$ and $g\in\hkbv(\Rbar^2)$.  Then for all $[a,b]\times[c,d]\subseteq\Rbar^2$ and all $(x,y)\in\Rbar^2$,
\begin{eqnarray*}
\left|\intabcd fg\right| & \leq & \norm{f}\left(4\norm{g}_\infty+2\norm{V_1g}_\infty +2\norm{V_2g}_\infty +V_{12}g
\right)\\
\left|\int_{-\infty}^x\int_{-\infty}^y fg\right| & \leq & \norm{f}
\left(\norm{g}_\infty+\norm{V_1g}_\infty +\norm{V_2g}_\infty +V_{12}g\right)=\norm{f}\norm{g}_{bv}.
\end{eqnarray*}
\end{prop}

We now give two equivalent norms.
\begin{prop}[Equivalent norms]
For $f\in\alextwo$ define $\norm{f}'=\sup_I\abs{\int_I f}$ where the supremum is taken over all
intervals $I\subseteq\Rbar^2$; $\norm{f}''=\sup_{g}\abs{\intinf \intinf fg}$ where the supremum is taken over all
$g\in\hkbv(\Rbar^2)$ such that $\norm{g}_{bv}\leq 1$.
\end{prop}
\begin{proof}
Since the characteristic function of an interval is of bounded variation integration by parts establishes
existence of $\norm{\cdot}'$. Clearly, $\norm{f}\leq\norm{f}'$.  And, there is the decomposition into
nonoverlapping intervals, 
$$
\intabcd f=\int_{-\infty}^a\int_{-\infty}^cf + \int_{-\infty}^b\int_{-\infty}^df
-\int_{-\infty}^a\int_{-\infty}^df-\int_{-\infty}^b\int_{-\infty}^cf,
$$
so that $\norm{f}'\leq 4\norm{f}$.

If $g\in\hkbv(\Rbar^2)$ with $\norm{g}_{bv}\leq 1$ then the H\"older inequality (Proposition~\ref{propholder}) 
establishes $\norm{f}''\leq \norm{f}$.  For a reverse inequality note there is
$(x,y)\in\Rbar^2$ such that $\abs{\int_{-\infty}^x\int_{-\infty}^y f}=\norm{f}$.
Let $g=(1/4)\chi_{[-\infty,x)\times[-\infty,y)}$.  Then $\norm{g}_{bv}=1$ (Example~\ref{examplegbv}). And,
$$
\norm{f}''\geq \left|\intinf fg\right|=
\frac{1}{4}\left|\int_{-\infty}^x\int_{-\infty}^y f\right|=
\frac{\norm{f}}{4}.
$$
Hence, $\norm{f}/4\leq\norm{f}''$.
\end{proof}

The integral $\intinf\intinf fg$ is not changed when
$g$ is changed on a coordinate line.
\begin{prop}\label{propbvline}
Let $f\in\alextwo$ and $g\in\hkbv(\Rbar^2)$.  
If $g$ is changed on a coordinate line the
integral $\intinf\intinf fg$ is not changed.
\end{prop}
Note that this includes the result that if $(s,t)\in\Rbar^2$ and $g=\chi_{(s,t)}$ then $\intinf\intinf fg=0$
for all $f\in\alextwo$.
\begin{proof}
Let $f\in\alextwo$ with primitive $F\in\balextwo$.

First show that if $g$ is the characteristic function of a point then $\intinf\intinf fg=0$.
Let $g=\chi_{(s,t)}$.  Then $\intinf\intinf fg=\intinf\intinf g\,d_{12}F$ by \eqref{hildebrandtA2} or
\eqref{hildebrandtB2}.
A gauge $\gamma$ can always be chosen so that if interval $I=[a,b]\times[c,d]$ is in a $\gamma$-fine
tagged division and $(s,t)\in I$ then its tag is $(s,t)$.
The only terms in a Riemann sum that do not necessarily
vanish are $g(s,t)[F(a,c)+F(b,d)-F(a,d)-F(b,c)]$ but the gauge can force this term to be arbitrarily
small due to the uniform continuity of $F$.  There can be at most four such terms.  Hence,
$\intinf\intinf fg=0$.

Now consider 
$$
g(x,y)=\left\{\begin{array}{cl}
\psi(x); & x\in\R, y=\infty\\
0; & \text{else}
\end{array}
\right.
$$
where $\psi\fn\Rbar\to\R$ is of bounded variation and $\psi(\pm\infty)=0$.
We again have $\intinf\intinf fg=\intinf\intinf g\,d_{12}F$.
Given a gauge $\gamma$ we can always choose a $\gamma$-fine
tagged division so that there is a point $y\in\R$ such that
if $I$ is an interval in the tagged division  then
$I=[x_{i-1},x_i]\times[y,\infty]$ with associated tag $(z_i,\infty)$ for which
$x_{i-1}\leq z_i\leq x_i$.  And, there is $N\in\N$ so that $-\infty=x_0<x_1<\ldots<x_N=\infty$.  We then
have a partition of the line $\{(s,\infty)\in\Rbar^2\mid s\in\Rbar\}$.  We can assume $z_1=-\infty$
and $z_N=\infty$.  The terms that do not necessarily
vanish in a Riemann
sum for such a division are
\begin{align*}
&\sum_{i=1}^N\psi(z_i)[F(x_{i-1},y)+F(x_i,\infty)-F(x_{i-1},\infty)-F(x_i,y)]\\
&=\sum_{i=1}^{N-1}\psi(z_i)[F(x_i,\infty)-F(x_i,y)]+\sum_{i=1}^{N-1}\psi(z_{i+1})[F(x_{i},y)-F(x_i,\infty)]\\
&=\sum_{i=1}^{N-1}[\psi(z_i)-\psi(z_{i+1})][F(x_{i},\infty)-F(x_i,y)].
\end{align*}
The Riemann sum is then bounded by $V\!\psi\,\sup_{x,y\in\R}\abs{F(x,\infty)-F(x,y)}$.  Since $F$ is uniformly
continuous this can be made arbitrarily small by choosing $\gamma$ to force $y$ close enough to $\infty$.
Hence, $\intinf fg=0$.  

Changing $g$ on other lines is handled similarly.
\end{proof}

To discuss the dual space of $\alextwo$ we need to define normalisations for functions of bounded variation.
Fix $\alpha_{--}, \alpha_{++}, \alpha_{-+}, \alpha_{+-}\in[0,1]$ such that 
$\alpha_{--}+ \alpha_{++}+ \alpha_{-+}+ \alpha_{+-}=1$.  If $g\in\hkbv(\Rbar^2)$ then 
define it's normalisation ${\tilde g}$ as follows.
For $(x,y)\in\R^2$ put ${\tilde g}(x,y)=\alpha_{--}\lim_{(s,t)\to(x-,y-)}g(s,t)+\alpha_{++}\lim_{(s,t)\to(x+,y+)}g(s,t)+
\alpha_{-+}\lim_{(s,t)\to(x-,y+)}g(s,t)+\alpha_{+-}\lim_{(s,t)\to(x+,y-)}g(s,t)$. This involves changing $g$ on a set
that is at most countable.
There is a similar procedure
on the boundary of $\R^2$.  For example, fix $\beta_-,\beta_+\in[0,1]$ such that $\beta_{-}+\beta_{+}=1$.
For $y\in\R$ we define 
${\tilde g}(\infty,y)=\beta_{-}\lim_{(s,t)\to(\infty,y-)}g(s,t)+\beta_{+}\lim_{(s,t)\to(\infty,y+)}g(s,t)$.  
A single limit is required at each of the four corner points of $\Rbar^2$.

Finally, we have a characterisation of the dual space of $\alextwo$.  It is clear from Proposition~\ref{propbvline}
that if two elements of the dual space differ only on a coordinate line then they represent the same dual space
element.  This is dealt with by fixing a normalisation on $\hkbv(\Rbar^2)$.
\begin{prop}[Dual space]
Fix any normalisation on $\hkbv(\Rbar^2)$.  The dual space of $\alextwo$ is $\alextwo'=\hkbv(\Rbar^2)$.
\end{prop}
\begin{proof}
The H\"older inequality shows that every function of bounded variation generates
a bounded linear functional on $\alextwo$ via $f\mapsto \intinf\intinf fg$ ($f\in\alextwo$,
$g\in\hkbv(\Rbar^2)$).  And, in \cite{lee}, Section~6.6, it is shown
that each element of the dual space of the Henstock--Kurzweil integrable functions is given by 
integration against a function of bounded
variation that vanishes on the boundary.  (This is done on a compact interval but the
proof extends immediately to $\Rbar^2$.)  Since the space of Henstock--Kurzweil integrable functions is
dense in $\alextwo$ (Proposition~\ref{propseparable}) this shows that the dual space of $\alextwo$ is also
$\hkbv(\Rbar^2)$.  By Proposition~\ref{propbvline} we get the same result for our given normalisation.
\end{proof}

Note that each normalisation on $\hkbv(\Rbar^2)$ gives an isometrically isomorphic representation of the dual space.
Equivalently, we can say the dual space is the set of functions of {\it essential bounded variation}.  This
is the set of equivalence classes of functions agreeing almost everywhere with a function of bounded
variation.  Choosing a normalisation merely selects one element of each equivalence class.
It is often misstated in the literature that the dual space of the continuous functions on the real line is
$\bv$ (including in \cite{talviladenjoy}) but the dual space is more properly given as 
normalised bounded variation
or essential bounded variation.
The formulas in Definition~\ref{defnparts} and Proposition~\ref{propgdF} are not defined if $g$ is of essential
bounded variation but the integral $\intinf fg$ can computed using sequences of $L^1$ functions as in
Proposition~\ref{propPhin}.

\section{Convergence theorems}\label{sectionconvergence}
A number of convergence theorems are given in \cite{ang} and \cite{talviladenjoy} that can be extended to 
$\alextwo$, including a necessary
and sufficient condition for interchanging limits and integrals.  The required changes are minor so we
do not present them here.  Instead, we give the convergence theorem that seems to be the most useful
in practice.  It refers to limits of $\intinf\intinf fg_n$ where $\{g_n\}$ is a sequence of functions of bounded
variation.  This can occur, for example, in a convolution product.  See \cite{talvilaheat} for an
application on the real line.
\begin{prop}\label{propconvergence}
Let $f\in\alextwo$.  Let $\{g_n\}\subset\hkbv(\Rbar^2)$ such that $\norm{g_n}_{bv}$ is bounded and
$\lim_{n\to\infty}g_n=g$ pointwise on $\Rbar^2$ for a function
$g\fn\Rbar^2\to\R$.  Then $g\in\hkbv(\Rbar^2)$ and $\lim_{n\to\infty}\intinf\intinf fg_n=\intinf\intinf fg$.
\end{prop}
\begin{proof}
We can write $\norm{g_n}_{bv}\leq M$.

To prove $g$ is bounded note that
$$
\abs{g(x,y)}\leq \abs{g(x,y)-g_n(x,y)}+\abs{g_n(x,y)}\leq \abs{g(x,y)-g_n(x,y)}+M.
$$
Now let $n\to\infty$.

Fix a finite collection of nonoverlapping intervals
$\{[a_i,b_i]\times[c_i,d_i]\}_{i=1}^N$.  We have the inequality
\begin{align}
&\sum_{i=1}^N\abs{g(a_i,c_i)+g(b_i,d_i)-g(a_i,d_i)-g(b_i,c_i)}\notag\\
&\qquad\leq\sum_{i=1}^N\abs{g(a_i,c_i)-g_n(a_i,c_i)}+\sum_{i=1}^N\abs{g(b_i,d_i)-g_n(b_i,d_i)}\label{convergence1}\\
&\qquad+\sum_{i=1}^N\abs{g(a_i,d_i)-g_n(a_i,d_i)}+\sum_{i=1}^N\abs{g(b_i,c_i)-g_n(b_i,c_i)}\label{convergence2}\\
&\qquad+\sum_{i=1}^N\abs{g_n(a_i,c_i)+g_n(b_i,d_i)-g_n(a_i,d_i)-g_n(b_i,c_i)}.\label{convergence3}
\end{align}
For these fixed finite sums we can take $n$ large enough so that the sums in \eqref{convergence1}
and \eqref{convergence2} contribute less than any prescribed $\epsilon>0$.  The sum in \eqref{convergence3}
is always less than $M$.  Hence, $V_{12}g<\infty$. 

To show $V_1g(\cdot,y_0)$ is finite for some $y_0\in\Rbar$ write
\begin{align*}
&\sum_{i=1}^N\abs{g(x_i,y_0)-g(x_{i-1},y_0)}\\
&\leq \sum_{i=1}^N\abs{g(x_i,y_0)-g_n(x_i,y_0)}+\sum_{i=1}^N\abs{g(x_{i-1},y_0)-g_n(x_{i-1},y_0)}
+V_1g_n(\cdot,y_0).
\end{align*}
As above, we can take $n$ large enough to make these last two sums small.  Similarly with $V_2g(x_0,\cdot)$.
Hence, $g\in\hkbv(\Rbar^2)$.

By linearity of the integral we can assume $g_n\to 0$.  Integration by parts then gives
\begin{eqnarray}
\intinf\intinf fg_n & = & F(\infty,\infty)g_n(\infty,\infty)-\intinf F(s,\infty)\,d_1g_n(s,\infty)\notag\\
 & & \qquad-\intinf F(\infty,t)\,d_2g_n(\infty,t)
+\intinf\intinf F(s,t)\,d_{12}g_n(s,t).\label{convergence}
\end{eqnarray}
The term $F(\infty,\infty)g_n(\infty,\infty)\to 0$.  To show the last integral above tends to $0$
we use the method in \cite[p.~126]{monteiroslaviktvrdy}.  Let $\epsilon>0$.  
Since $F$ is uniformly continuous, we can take a gauge $\gamma$ so that for each interval $I_i$ in a
$\gamma$-fine tagged division, if $(x,y)$ and $(s,t)$ are points in $I_i$ then $\abs{F(x,y)-F(s,t)}
<\epsilon$.  Now suppose $\{[a_i,b_i]\times[c_i,d_i],(x_i,y_i)\}_{i=1}^N$ is a $\gamma$-fine tagged
division of $\Rbar^2$.  Let $\Delta_ig_n=g_n(a_i,c_i)+g_n(b_i,d_i)-g_n(a_i,d_i)-g_n(b_i,c_i)$ and 
$I_i=[a_i,b_i]\times[c_i,d_i]$.  Then
\begin{align*}
&\left|\intinf\intinf F(x,y)\,d_{12}g_n(x,y)-\sum_{i=1}^NF(x_i,y_i)\Delta_ig_n\right|\\
&=\left|\sum_{i=1}^N\left[\int_{I_i}F(x,y)\,d_{12}g_n(x,y)-F(x_i,y_i)\int_{I_i}d_{12}g_n(x,y)\right]\right|\\
&=\left|\sum_{i=1}^N\int_{I_i}\left[F(x,y)-F(x_i,y_i)\right]\,d_{12}g_n(x,y)\right|\\
&\leq \epsilon V_{12}g_n\leq \epsilon M.
\end{align*}
Therefore, for a fixed tagged division the Riemann sums approximate the integral uniformly in $n$.
But the Riemann sums tend to $0$ as $n\to\infty$ since $g_n\to 0$.  Similarly, for the other two
integrals in \eqref{convergence}.
\end{proof}
Note that we also get convergence on every subinterval of $\Rbar^2$.

\section{Change of variables}\label{sectionchange}
If $V$ and $W$ are open sets in $\R^n$ a typical change of variables theorem for $L^1$ functions is
that $\int_Wf\,d\lambda=\int_V(f\circ T)\abs{\det J_T}\,d\lambda$, where
$T\fn V\to W$ is a diffeomorphism, $J_T$ is the Jacobian and $f\in L^1(W)$.  For a proof see \cite{folland}.

For the continuous primitive integral on the real line the following theorem appears in 
\cite{talviladenjoy}:
\begin{theorem}
Suppose $f\in\alex$ and $F'=f$ where $F\in C(\Rbar)$.  Let
$-\infty\leq a<b\leq\infty$. If
$G\in C([a,b])$  then 
$$
\int_{G(a)}^{G(b)} f  =  \int_a^b (f\circ G)\,G'=
(F\circ G)(b) -(F\circ G)(a).
$$
If $G\in C((a,b))$ and $\lim_{t\to a^+}G(t)=-\infty$ and
$\lim_{t\to b^-}G(t)=\infty$ then 
$$
\intinf f = \int_a^b(f\circ G)\,G'
=F(\infty)-F(-\infty).
$$
\end{theorem}
The function $F\circ G$ is continuous so its continuous primitive integral exists.
The quantity $(f\circ G)\,G'$ is written in place of $(F\circ G)'$ and it
is shown in \cite{talviladenjoy} that if two sequences of differentiable functions
converge to $F$ and $G$, respectively, then the usual pointwise formula for differentiation of
a composite function converges in the Alexiewicz norm to $(F\circ G)'$.  However, this does
not imply separate existence of $f\circ G$ and $G'$ and the multiplication is purely formal.
Indeed, suppose $F(x)=x^2$ and $G$ is continuous but pointwise differentiable nowhere.  Then
$F'\circ G=2G$.  An arbitrary distribution can be multiplied by a $C^\infty$ function and
a distribution in $\alex$ can be multiplied by a function of bounded variation but $G$ is not
of bounded variation so the product $2GG'$ has no meaning other than shorthand for $(G^2)'$.

Here we choose to present a restricted change of variables formula that has immediate application
to convolutions (Section~\ref{sectionconvolution}).

There is a well-established method of composing a distribution with a linear bijection.
Suppose $\Psi\fn\R^2\to\R^2$ is a linear bijection.  For a distribution $T\in\Dptwo$ the
composition $T\circ \Psi\in\Dptwo$ is defined for $\phi\in\Dtwo$ by 
$\langle T\circ\Psi,\phi\rangle=[\det\Psi]^{-1}
\langle T,\phi\circ\Psi^{-1}\rangle$.  For example, \cite[p.~285]{folland}.

If $f\in\alextwo$ then integration of $f\circ\Psi$ requires $\Psi$ to map intervals onto
finite unions of intervals (in the same Cartesian coordinate system)
 since these are the only regions in $\Rbar^2$ for which the
integral is defined.  This can be accomplished by having each component of $\Psi$ depend
linearly on only one variable.

\begin{theorem}\label{theoremchange}
Let $\alpha, \beta,\gamma_1, \gamma_2\in\R$ such that $\alpha\beta\not=0$.
Let $[a,b]\times[c,d]\subseteq\R^2$ and let $f\in\alextwo$.
(a) If $\Psi\fn\R^2\to\R^2$ is given by $\Psi(u,v)=(\alpha u+\gamma_1,\beta v+\gamma_2)$
then
$$\intabcd f(x,y)\,dy\,dx=\alpha\beta\int_{(a-\gamma_1)/\alpha}^{(b-\gamma_1)/\alpha}
\int_{(c-\gamma_2)/\beta}^{(d-\gamma_2)/\beta} f(\alpha u+\gamma_1,\beta v+\gamma_2)\,dv\,du.
$$
(b) If $\Psi\fn\R^2\to\R^2$ is given by $\Psi(u,v)=(\beta v+\gamma_2,\alpha u+\gamma_1)$
then
$$\intabcd f(x,y)\,dy\,dx=\alpha\beta
\int_{(c-\gamma_1)/\alpha}^{(d-\gamma_1)/\alpha}
\int_{(a-\gamma_2)/\beta}^{(b-\gamma_2)/\beta}
f(\beta v+\gamma_2,\alpha u+\gamma_1)\,dv\,du.
$$
\end{theorem}
The proof follows from the above definition for composition with a linear bijection.
Note that $\Psi$ maps intervals to intervals, as indicated by the limits of integration
on the integrals in the theorem.

The convention following Definition~\ref{defnctsprimitiveintegral} 
on ordering of upper and lower limits of integration has been used.
If any of $a,b,c,d$ is in $\{\infty,-\infty\}$ then the usual arithmetic of infinities can
be used to determine the limits of integration.  For example, if $a=-\infty$ then
replace $(a-\gamma_1)/\alpha$ with $\sgn(-\alpha)\infty$.

See \cite{mawhin} for similar change of variables for the Henstock--Kurzweil integral.

\section{Banach lattice}\label{sectionbanachlattice}
The usual pointwise ordering on $\balextwo$ makes it into a Banach lattice and $\alextwo$
inherits this structure.  This creates a distributional ordering such that all distributions in $\alextwo$ 
have absolutely convergent integrals.  For functions in $\alextwo$ the usual pointwise ordering
leads to conditionally convergent integrals.  See the second paragraph of Section~\ref{sectionexamples}.

This distributional ordering has been used to solve problems in
ordinary and partial differential equations.  See 
S.~Heikkil\"a \cite{heikkila2011a}, \cite{heikkila2011b}, \cite{heikkila2012a}, \cite{heikkila2012b};
S.~Heikkil\"a and E. Talvila \cite{heikkilatalvila};
Liu Wei and Ye Guoju with numerous co-authors, for example,
\cite{liuAnGuojo}.

A reference for Banach lattices is \cite{aliprantisborder}.
The definitions in this section are largely repeated from \cite{talvilaacrn}.
Corresponding lattice results were obtained for distributional integrals
on the real line with continuous primitives \cite{talviladenjoy}, with regulated primitives
\cite{talvilaregulated}, and of higher order \cite{talvilaacrn}.  We omit proofs in this
section since they are so similar to results in these papers.

If $\preceq$ is a binary
operation on set $S$ then it is a {\it partial order} if for all
$x,y,z\in S$ it is {\it reflexive} ($x\preceq x$), {\it antisymmetric}
($x\preceq y$ and $y\preceq x$ imply $x=y$) and {\it transitive} ($x\preceq y$
and $y\preceq z$ imply $x\preceq z$). 
If
$S$ is a Banach space with norm $\norm{\cdot}_S$ and $\preceq$ is a partial
order on $S$  then $S$ is a {\it Banach lattice} if for all $x,y\in S$
\begin{enumerate}
\item
$x\vee y$ and $x\wedge y$ are in $S$.  The {\it join} is 
$x\vee y=
\sup\{x,y\}=w$ such that $x\preceq w$, $y\preceq w$ and if $x\preceq\tilde{w}$
and $y\preceq\tilde{w}$ then $w\preceq\tilde{w}$.
The {\it meet} is 
$x\wedge y=
\inf\{x,y\}=w$ such that $w\preceq x$, $w\preceq y$ and if $\tilde{w}\preceq x$
and $\tilde{w}\preceq y$ then $\tilde{w}\preceq w$.
\item
$x\preceq y$ implies $x+z\preceq y+z$ for all $z\in S$.
\item
$x\preceq y$ implies $kx\preceq ky$ for all $k\in\R$ with $k\geq 0$.
\item
$\abs{x}\preceq \abs{y}$ implies $\norm{x}_S\leq \norm{y}_S$.
\end{enumerate}
If $x\preceq y$ we write $y\succeq x$.
We also define
$\abs{x}=x\vee (-x)$, $x^+=x\vee 0$ and $x^-=(-x)\vee 0$. 
Then $x=x^+-x^-$ and $\abs{x}=x^++x^-$.

The usual pointwise ordering, $F_1\leq F_2$ if and only if $F_1(x,y)\leq F_2(x,y)$
for all $(x,y)\in\Rbar^2$,  is a partial order on $\balextwo$.
Since $\balextwo$ is closed under the operations
$(F_1\vee F_2)(x,y)=\sup(F_1,F_2)(x,y)=\max(F_1(x,y),F_2(x,y))$ and
$(F_1\wedge F_2)(x,y)=\inf(F_1,F_2)(x,y)=\min(F_1(x,y),F_2(x,y))$,
it is then a vector lattice (or Riesz space).  Since $\abs{F_1}\leq\abs{F_2}$ implies
$\norm{F_1}_\infty\leq\norm{F_2}_\infty$ we have that $\balextwo$ is a Banach lattice.

A partial ordering in $\alextwo$ is inherited from $\balextwo$.  If $f_1, f_2\in
\alextwo$ with respective primitives $F_1, F_2\in\balextwo$ then  $f_1\preceq f_2$
if and only if $F_1\leq F_2$ in $\balextwo$.  The isomorphism between $\alextwo$ and $\balextwo$
now shows $\alextwo$ is also a Banach lattice.

It is not a linear ordering.  For example, if $F(x,y)=\exp(-x^2-y^2)$ and $G(x,y)=\exp(-(x-1)^2-(y-1)^2)$
then we have neither $F\leq G$ nor $G\leq F$ and similarly in $\alextwo$.

An element $e\geq 0$ such that for each $x\in S$ there is $\lambda>0$
such that $\abs{x}\leq \lambda e$ is an {\it order unit} for lattice 
$S$.  The order unit for $\balextwo$ would have to vanish on $\{-\infty\}\times\Rbar$
and on $\Rbar\times\{-\infty\}$ and decay to $0$ more slowly than any continuous function.
(See \cite[Theorem~5.1]{talvilaacrn}).
Hence $\alextwo$ does not have an
order unit.

We have absolute
integrability: if $f\in\alextwo$ so  is $\abs{f}$.  The partial derivative operator
$\dalex$ commutes with $\vee$ and $\wedge$ and hence with $\abs{\cdot}$.

\begin{theorem}[Banach lattice]\label{theoremlattice}
(a) $\balextwo$ is a Banach lattice.
(b)
For $f_1,f_2\in\alextwo$ with respective primitives $F_1,F_2\in\balextwo$, 
define $f_1\preceq f_2$ if $F_1\leq F_2$ in $\balextwo$.
Then $\alextwo$ is a Banach lattice isomorphic to $\balextwo$.
(c)
$\balextwo$ and $\alextwo$ do not have an order unit.
(d)  Let $F_1,F_2\in\balextwo$.  Then $\dalex(F_1\vee F_2)=
(\dalex F_1)\vee(\dalex F_2)$,
$\dalex(F_1\wedge F_2)=(\dalex F_1)\wedge(\dalex F_2)$, $\abs{\dalex F}=\dalex\abs{F}$,
$\dalex(F^+)=(\dalex F)^+$, and $\dalex(F^-)=(\dalex F)^-$.
(e) If $f\in\alextwo$ with primitive $F\in\balextwo$ then $\abs{f}\in\alextwo$ with 
primitive $\abs{F}\in\balextwo$ and $\int_{-\infty}^x\int_{-\infty}^y \abs{f}=
\abs{\int_{-\infty}^x\int_{-\infty}^y f}$ for all $(x,y)\in\Rbar^2$. 
And, $\norm{\abs{f}}=\norm{f}$, 
$\norm{f^{\pm}}\leq \norm{f}$.  
(f) If $f\in\alextwo$ then $f^{\pm}\in\alextwo$ with respective
primitives $F^{\pm}\in\balextwo$.  {\it Jordan decomposition}: $f=f^+ - f^-$.
And, $\intinf fg=\intinf f^+g-\intinf f^-g$ for every $g\in\hkbv(\Rbar^2)$.
(g) $\alextwo$ is {\it distributive}:
$f_1\wedge(f_2\vee f_3)=(f_1\wedge f_2)\vee(f_1\wedge f_3)$
and  $f_1\vee(f_2\wedge f_3)=(f_1\vee f_2)\wedge(f_1\vee f_3)$ for all $f_1,f_2,f_3\in\alextwo$. 
(h) $\alextwo$ is {\it modular}: For all $f_1,f_2\in\alextwo$, if
$f_1\preceq f_2$ then $f_1\vee(f_2\wedge f_3)=f_2\wedge(f_1\vee f_3)$ for all $f_3\in\alextwo$.
(i) Let $F_1$ and $F_2$ be continuous functions on $\Rbar^2$. Then 
\begin{eqnarray*}
\dalex F_1\preceq \dalex F_2 & \Longleftrightarrow & F_1(-\infty,-\infty)+F_1(x,y)-F_1(-\infty,y)
-F_2(x,-\infty)\\
 & & \qquad\leq 
F_2(-\infty,-\infty)+F_2(x,y)-F_2(-\infty,y)
-F_2(x,-\infty)
\end{eqnarray*}
for all $(x,y)\in\Rbar^2$.

\end{theorem}

Let $f_1,f_2\in\alextwo$ with respective primitives $F_1,F_2\in\balextwo$.
Note that if $F_1\leq F_2$ in $\balextwo$ then we can differentiate both
sides of this inequality with $\dalex$ to get  $f_1\preceq f_2$ in $\alextwo$.  And,
if $f_1\preceq f_2$ in $\alextwo$ we can integrate both sides against 
$\chi_{(-\infty,x)\times(-\infty,y)}$ to get $F_1\leq F_2$ in $\balextwo$.  See Theorem~\ref{propftc}. 
This also shows the derivative $\dalex$ is a positive operator on 
$\balextwo$ and its inverse is a positive operator on $\alextwo$.

In general, $\intabcd \abs{f}$ and $\abs{\intabcd f}$ are not comparable.  However,
if $a=-\infty$ or $c=-\infty$ then $\intabcd \abs{f}\leq\abs{\intabcd f}$.

The usual pointwise ordering makes $L^1$ into a Banach lattice.
But the space of Henstock--Kurzweil integrable functions is not
a vector lattice. It is not closed under supremum and infimum since there
are functions integrable in this sense for which $\intinf\intinf f(x,y)\,dy\,dx$
converges but $\intinf \abs{f(x,y)}\,dy\,dx$ diverges.  For example, 
the function $f(x,y)=\sin(x)\sin(y)/(xy)$ from Section~\ref{sectionexamples}.
Thus, even for functions,
when we allow conditional convergence we must look elsewhere to find
a lattice structure.

Consider the 
example
$$
F(x,y)=\left\{\begin{array}{cl}
\int_0^x\int_0^y\frac{\sin(s)\sin(t)}{st}dt\,ds, & x, y\geq 0\\
0, & \text{otherwise}.
\end{array}
\right.
$$
Then $F(x,y)\geq 0$ 
for all $(x,y)\in\Rbar^2$ so $\dalex F\succeq 0$ in
$\alextwo$.  The integrand is not positive in a pointwise sense so $\preceq$ is not compatible with 
the usual pointwise ordering on $\alextwo$.  The
order $\preceq$ is also not compatible with the usual order on distributions:
if $T,U\in\Dp$ then $T\geq U$ if and only if $\langle T-U,\phi\rangle
\geq 0$ for all $\phi\in\D$ such that $\phi\geq 0$.  If $T\geq 0$
then it is known that $T$ is a Borel measure.  The usual ordering on
distributions does
not give a vector lattice on $\alextwo$. With the distributional ordering, $\sup(\dalex F,0)$ is the function
equal to $\sin(x)\sin(y)/(xy)$ when $x\in[2m\pi, (2m+1)\pi]$ and $y\in[2n\pi, (2n+1)\pi]$, 
or, $x\in[(2m+1)\pi, (2m+2)\pi]$ and $y\in[(2n+1)\pi, (2n+2)\pi]$ for some integers
$m, n\geq 0$, and is equal to
$0$ otherwise.  This function is not in $\alextwo$ since the integral
defining $F$ converges conditionally.
The derivative $\dalex F$ is not positive in
the pointwise  or
distributional sense.  Note that in $\alextwo$ we have
$(\dalex F)^+=\abs{\dalex F}=\dalex F$ and $(\dalex F)^-=0$.

A vector lattice is {\it order complete} (or {\it Dedekind complete})
if every nonempty subset that is bounded
above has a supremum.  But $\balextwo$ is not complete.
Let $F_n(x,y)=\abs{\sin(\pi/x)\sin(\pi/y)}$ for $x,y\geq 1/n$ with $F_n(x,y)=0$ if $x\leq 1/n$ or $y\leq 1/n$.
Let $S=\{F_n\mid n\in\N\}$ then $S\subseteq\balextwo$.
An upper bound for $S$ is the function 
$$
F(x,y) =\left\{\begin{array}{cl}
1, & x,y>0\\
\frac{1}{\abs{x}+1}, & x<0,y>0\\
\frac{1}{\abs{y}+1}, & x>0,y<0\\
\frac{1}{(\abs{x}+1)(\abs{y}+1)}, & x,y<0.\\
\end{array}
\right.
$$
But $\sup(S)(x,y)=\chi_{(0,\infty)\times(0,\infty)}(x,y)\abs{\sin(\pi/x)\sin(\pi/y)}$,
which is not continuous.
Hence, $\alextwo$ is also not complete.

A vector lattice is {\it Archimedean} if whenever $0\leq x\leq ny$ for
all $n\in\N$ and some $y\geq 0$ then $x=0$.  Applying the Archimedean
property at each point of the domain $\Rbar^2$ shows $\balextwo$ and hence $\alextwo$ are Archimedean.
All lattice inequalities that hold in $\R$ also hold in all Archimedean spaces
and all lattice equalities that hold in $\R$ also hold in all vector lattices.
See \cite{aliprantisborder}.  This expands the list of identities and
inequalities proved in Theorem~\ref{theoremlattice}.

A Banach lattice is an {\it abstract $L$-space} if
$\norm{x+y}=\norm{x}+\norm{y}$ for all $x,y\geq 0$.  
A Banach lattice is an {\it abstract $M$-space} if
$\norm{x\vee y}=\max(\norm{x},\norm{y})$ for all $x,y\geq 0$.  
See, for example, \cite{aliprantisborder}.  We next
show that $\balextwo$ and $\alextwo$ are abstract $M$-spaces but neither
is an abstract $L$-space.
\begin{theorem}
Both of $\balextwo$ and $\alextwo$ are abstract $M$-spaces.
Neither is an abstract $L$-space.
\end{theorem}
For a proof see \cite[Theorem~5.2]{talvilaacrn}.

For every measure $\mu$ it is
known that $L^1(\mu)$ is an abstract $L$-space and that a Banach
lattice is an abstract $L$-space if and only if it is lattice
isometric to  
$L^1(\nu)$ for some
measure $\nu$.  Notice that $L^\infty(\mu)$ is an abstract $M$-space.
A Banach lattice is an abstract $M$-space with unit  
if and only if it is lattice isometric to $C(K)$ for some compact
Hausdorff space $K$.
These results are due to S.~Kakutani, M.~Krein and
others.  For references see \cite{aliprantisborder}.
In our case, $\balextwo$ and $\alextwo$ are isomorphic to the set of
continuous functions that vanish on $\{-\infty\}\times\Rbar$
and on $\Rbar\times\{-\infty\}$.  It is not clear what the space $K$
is here.
The fact that $\alextwo$ is an abstract $M$-space but not an abstract
$L$-space indicates that what we have termed an integral here is
fundamentally different from the Lebesgue integral.  

\section{Banach algebra}\label{sectionbanachalgebra}
A {\it commutative algebra} is a vector space $V$ over scalar field $\R$
with a multiplication $V\times V\mapsto V$ such that for all
$u,v,w\in V$ and all $a\in\R$, $u(vw) =(uv)w$ (associative),
$uv=vu$ (commutative),
$u(v+w)=uv+uw$ and $(u+v)w=uw+vw$ (distributive),
$a(uv)=(au)v$. 
If $(V,\norm{\cdot}_V)$ is a Banach space and $\norm{uv}_V\leq
\norm{u}_V\norm{v}_V$ then it is a Banach algebra.
For any compact Hausdorff space, $K$, the set of continuous
real-valued functions $C(K)$ is a commutative Banach algebra
under pointwise multiplication and
the uniform norm.  Since $\Rbar^2$ is compact and $\balextwo$ is
closed under pointwise multiplication, $\balextwo$ is a subalgebra of
$C(\Rbar^2)$.
The usual pointwise multiplication, $(FG)(x,y)=[F(x,y)][G(x,y)]$
for all $(x,y)\in\Rbar^2$, then makes $\balextwo$ into a commutative algebra.
The inequality
$\norm{F_1F_2}_\infty\leq\norm{F_1}_\infty\norm{F_2}_\infty$ for
all $F_1,F_2\in\balextwo$ shows
$\balextwo$ is a commutative Banach algebra.  

There is no unit.  For suppose $F(x,y)>0$ for all $(x,y)\in\R^2$.  If $eF=F$ 
then $e(x,y)=1$ for all $(x,y)\in\R^2$ so $e\not\in\balextwo$.

Consider the sequence $(u_n)\subset\balextwo$ defined by $u_n(x)=
0$ for $x\leq -n$, $u_n(x)=x+n$ for $-n\leq x\leq 1-n$ and
$u_n(x)=1$ for $x\geq 1-n$.  Define $U_n\in\balextwo$ by $U_n(x,y)=u_n(x)u_n(y)$.  For each $F\in\balextwo$ we have
$\norm{F-U_nF}_\infty\to 0$.  Given $\epsilon>0$ there is $M\in\R$
such that $\abs{F(x,y)}<\epsilon$ for all $(x,y)$ such that $x\leq M$ or $y\leq M$.  We then have
$\abs{F(x,y)-U_n(x,y)F(x,y)}=\abs{F(x,y)}\abs{1-U_n(x,y)}<\epsilon$ if $x\leq M$ or $y\leq M$.
If $x\geq M$ and $y\geq M$ take $n\geq 1-M$.  Then $U_n(x,y)=1$.  Hence,
$\norm{F-U_nF}_\infty\to 0$.  $\balextwo$ is then said to have an
{\it approximate identity}.

\begin{theorem}
If $f_1,f_2\in\alextwo$ with respective primitives $F_1,F_2\in\balextwo$ define their product by $f_1f_2=\dalex(F_1F_2)$.
Then $\alextwo$ is a commutative Banach algebra without
unit, with approximate identity, isomorphic to $\balextwo$.
\end{theorem}

There is no difficulty in allowing functions in $\balextwo$ to be
complex-valued and using $\C$ as the field of scalars.  
Complex conjugation is then an involution on $\balextwo$.  Then
$\balextwo$ is a $C^\ast$-algebra since for each $F\in\balextwo$ we have
$\norm{{\overline F}}_\infty =\norm{F}_\infty$ and
$\norm{F{\overline F}}_\infty=\norm{F}_\infty^2$.  Thus,
$\alextwo$ is also a  $C^\ast$-algebra.

Suppose $f_1,f_2\in\alextwo$ have respective primitives $F_1,F_2\in\balextwo$.
Let $g\in\hkbv(\Rbar^2)$. Then according to Definition~\ref{defnparts}
\begin{align*}
&\int_{-\infty}^x\int_{-\infty}^y (f_1f_2)g  =  F_1(x,y)F_2(x,y)g(x,y)-\int_{-\infty}^x F_1(s,\infty)F_2(s,\infty)\,d_1g(s,\infty)\\
&\qquad-\int_{-\infty}^y F_1(\infty,t)F_2(\infty,t)\,d_2g(\infty,t)
+\int_{-\infty}^x\int_{-\infty}^y F_1(s,t)F_2(s,t)\,d_{12}g(s,t).
\end{align*}

There are zero divisors.  Let $F_1, F_2\in\balextwo$ with disjoint
supports.  Then $F_1F_2=0$ in $\balextwo$ so $\dalex(F_1F_2)=0$ in
$\alextwo$, yet neither $\dalex F_1$ nor $\dalex F_2$ need be zero.
This example also shows the multiplication introduced in $\alextwo$
is not compatible with pointwise multiplication in the case when
elements of $\alextwo$ are functions.

The product of a function in $\balextwo$ and a function in $C(\Rbar^2)$ is
in $\balextwo$.  Therefore,
$\balextwo$ is an ideal of $C(\Rbar^2)$.  The maximal ideals of $C(\Rbar^2)$
consist of functions vanishing at a single point.  See, for example,
\cite{kadison} and
\cite{kaniuth} for this and other results that hold for continuous functions
on a compact Hausdorff space (and hence for $\alextwo$).

\section{Iterated integrals}\label{sectioniteratedintegrals}
It was shown in \cite[Theorem~4]{ang} that if $f\in\alextwo$ then a type of Fubini
theorem holds in that
$$
\intabcd f= \int_a^b\left(\int_c^d f(x,y)\,dy\right)dx=\int_c^d\left(\int_a^b f(x,y)\,dx\right)dy 
$$
and the integral over an interval in $\Rbar^2$ is equal to the two iterated integrals.

A sufficient condition for existence of the iterated integrals, that can sometimes take
the place of Tonelli's theorem
in $\alextwo$, is the following.
\begin{prop}\label{propfubini1}
Let $f\in\alex$.
Let $g\fn\R^2\to\R$ be measurable.
Assume (i) for each $x\in\R$ the function
$y\mapsto g(x,y)$  is in $\bv(\R)$; (ii) the function $x\mapsto
V_{2}g(x,\cdot)$ is in $L^1(\R)$; (iii) there is $M\in L^1(\R)$
such that for each $y\in \R$ we have $\abs{g(x,y)}\leq M(x)$.  Then
the iterated integrals exist and are equal,
$\intinf\intinf f(y)g(x,y)\,dy\,dx=
\intinf\intinf f(y)g(x,y)\,dx\,dy$.
\end{prop}
For a proof see \cite[Proposition~A.3]{talvilaconv}.
The proposition was first proved for the wide Denjoy integral on compact intervals on page~$58$ in
\cite{celidze}. 

Calculus and integration texts often contain examples of functions of two variables 
for which the iterated integrals
are not equal.  These conundrums can usually be resolved by showing the primitive is not continuous.
\begin{example}
Let $\Omega\subset\Rbar^2$ be the interval
$\Omega=\{(x,y)\in\Rbar^2\mid 0<x< 1, 0< y<\infty\}$.
Let $F\fn\Omega\to\R$ be given by $F(x,y)=x^y$.  Then $F$ is continuous on $\Omega$ but there is no
way to extend the domain of $F$ to
${\overline \Omega}$ so that $F$ is continuous.  For, we have the limiting values,
$F(0,y)=0$ for $0<y<\infty$, $F(1,y)=1$ for $0< y<\infty$,
$F(x,0)=1$ for $0<x< 1$, $F(x,\infty)=0$ for $0<x<1$.  Hence, $F$ cannot be made continuous on
${\overline \Omega}$.
Now we let $f(x,y)=\dalex F(x,y)=x^{y-1}+x^{y-1}y\log(x)$ for $(x,y)\in\Omega$.
Since $F$ is not continuous on $\Omega$ the integral $\int_\Omega f$ does not exist, 
yet the two iterated integrals are equal. 
A calculation shows that for each $0<x<1$ we have $\int_0^\infty f(x,y)\,dy=0$ so
$\int_0^1\left(\int_0^\infty f(x,y)\,dy\right)dx=0$.  For each $0<y<\infty$ we have
$\int_0^1 f(x,y)\,dx=0$ so
$\int_0^\infty\left(\int_0^1 f(x,y)\,dx\right)dy=0$.
Suppose $0<a<b<1$, $0<c<d<\infty$.  Taking iterated limits
\begin{eqnarray*}
\lim_{b\to 1^-}\lim_{d\to\infty}\lim_{a\to 0^+}\lim_{c\to 0^+}\intabcd f   & = & 
\lim_{b\to 1^-}\lim_{d\to\infty}\lim_{a\to 0^+}\lim_{c\to 0^+}\left[a^c+b^d-a^d-b^c\right]\\
 & = & \lim_{b\to 1^-}\lim_{d\to\infty}\left[1+b^d-0-1\right] =0
\end{eqnarray*}
and
$$
\lim_{b\to 1^-}\lim_{d\to\infty}\lim_{c\to 0^+}\lim_{a\to 0^+}\intabcd f  =-1.
$$
Hence, $\int_\Omega f$ does not exist.
\end{example}
\begin{example}
Let $F(x,y)=\arctan(xy)$.  Then
$$
F_1(x,y)=\frac{y}{x^2y^2+1},\,\, F_2(x,y)=\frac{x}{x^2y^2+1},\,\, F_{12}(x,y)=\frac{1-x^2y^2}{(x^2y^2+1)^2}=F_{21}(x,y).
$$
We have the iterated improper Riemann integrals
\begin{eqnarray*}
\intinf\int_0^1 F_{12}(x,y)\,dy\,dx & = & \intinf\left[F_1(x,1)-F_1(x,0)\right]\,dx\\
 & = & \intinf \frac{dx}{x^2+1}=\pi\\
\int_0^1\intinf F_{21}(x,y)\,dx\,dy & = & \int_0^1\left[\lim_{x\to\infty}F_2(x,y)-\lim_{x\to-\infty}F_2(x,y)
\right]\,dy\\
 & = & \int_0^1 0\,dy=0.
\end{eqnarray*}
Although $F$ is bounded and continuous on $\R^2$, it is not continuous on $\Rbar^2$.
This can be seen by examining the behaviour of $F(x,y)$ in a neighbourhood of
the point $(0,\infty)$.  Hence, $\intinf\intinf \dalex F$ does not exist.
\end{example}
In $\R^2$ the iterated integrals theorem takes the following form.
\begin{prop}\label{propfubini2}
Let $f\in\alextwo$.
Let $g\fn\Rbar^2\times\Rbar^2\to\R$ be measurable on $\R^2\times\R^2$.
Assume (i) for each $(x,y)\in\Rbar^2$ the function
$(s,t)\mapsto g(x,y;s,t)$  is in $\hkbv(\Rbar^2)$; 
(ii) for each $t\in\Rbar$ the function $(x,y)\mapsto V_1g(x,y;\cdot,t)\in L^1(\R^2)$,
for each $s\in\Rbar$ the function $(x,y)\mapsto V_2g(x,y;s,\cdot)\in L^1(\R^2)$,
the function $(x,y)\mapsto V_{12}g(x,y;\cdot,\cdot)\in L^1(\R^2)$;
(iii) there is $M\in L^1(\R^2)$
such that for each $(s,t)\in \Rbar^2$ we have $\abs{g(x,y;s,t)}\leq M(x,y)$.  Then
the iterated integrals exist and are equal,
$\intinf\intinf f(s,t)g(s,t;x,y)\,dt\,ds\,\,dy\,dx=
\intinf\intinf f(s,t)g(s,t;x,y)\,dy\,dx\,dt\,ds$.
\end{prop}
Note that the variation in (ii) is computed with respect to
the second pair of variables in $g$, while the integration in (ii) and (iii) is computed
with respect to the first pair of variables.  The proof is similar to that of 
Proposition~\ref{propfubini1}.  The final step uses the density of step functions in
$\balextwo$ (Theorem~\ref{propseparable}).

\section{Convolution}\label{sectionconvolution}
In this section the convolution  $f\ast g$ is defined for $f\in\alextwo$ and $g\in\hkbv(\Rbar^2)$
and then for $g\in L^1(\R^2)$.

In Theorem~\ref{theoremconvolutionacbv} it is shown that when $g\in\hkbv(\Rbar^2)$ the convolution has similar
properties to the case when $f\in L^1$ and $g\in L^\infty$.  Since $L^\infty$ is the dual
space of $L^1$ this mirrors the fact that $\hkbv(\Rbar^2)$ is the dual space of $\alextwo$.  In
Theorem~\ref{theoremconvolutionacL1} the density of $L^1(\R^2)$ in $\alextwo$ is used to
define the convolution for $f\in\alextwo$ and $g\in L^1(\R^2)$.  This type of convolution
has properties analogous to convolutions on $L^1\times L^1$.

Convolutions in $\alex$ were introduced in \cite{talvilaconv}.  Here we extend the two most
important theorems from $\R$ to $\R^2$.  Many other results, such as differentiation
and integration of convolutions, can also be carried over to $\R^2$.

First we show the convolution is well-defined.
Fix $(x,y)\in\R^2$.  If $f\in\alextwo$ has primitive $F\in\balextwo$ define
$\Phi(s,t)=F(x-s, y-t)$.  Then $\Psi\in\balextwo$ and $\dalex \Psi(s,t)=\dalex F(x-s,y-t)$.
We can then define $\psi(s,t)=f(x-s,y-t)=\dalex\Psi(s,t)$.  Then 
$f\ast g(x,y)=\intinf\intinf f(x-s,y-t)g(s,t)\,dt\,ds$ is well-defined for each $g\in\hkbv(\Rbar^2)$.
See Theorem~\ref{theoremchange}.

\begin{theorem}\label{theoremconvolutionacbv}
Let $f\in\alextwo$, let $F\in\balextwo$ be its primitive 
and let $g\in\hkbv(\Rbar^2)$. Then (a) $f\ast g$  exists on $\R^2$
(b) $f\ast g=g\ast f$  (c) $\norm{f\ast g}_\infty \leq \norm{f}\norm{g}_{bv}$
(d) $f\ast g\in C(\Rbar^2)$.  Let $\epsilon_1,\epsilon_2\in\{+,-\}$.  Then
$\lim_{\substack{x\to\epsilon_1\infty\\y\to\epsilon_2\infty}}
f\ast g(x)=g(\epsilon_1\infty,\epsilon_2\infty)F(\infty,\infty)$.
(e) If $h\in L^1(\R^2)$ then $f\ast (g\ast h)=(f\ast g)\ast h\in C(\Rbar^2)$.
\end{theorem}
\begin{proof}
(a) The above definition and integration by parts show $f\ast g$ exists on $\R^2$.
(b) If $g\in\hkbv(\Rbar^2)$ then the function $(s,t)\mapsto g(x-s,y-t)$ is also
in $\hkbv(\Rbar^2)$.  Hence, $g\ast f$ exists in $\R^2$.  We can change variables as in 
Theorem~\ref{theoremchange}.
(c) This follows from Proposition~\ref{propalexinvariance} and the H\"older inequality
(Proposition~\ref{propholder}).
(d) To show continuity at $(x,y)\in\R^2$, let $(\xi,\eta)\in\R^2$.  Then
$$
\abs{f\ast g(x,y)-f\ast g(\xi,\eta)}  \leq  \norm{f(x-\cdot,y-\cdot)-f(\xi-\cdot,\eta-\cdot)}\norm{g}_{bv}.
$$
This last expression tends to $0$ as
$(\xi,\eta)\to (x,y)$ by continuity in the Alexiewicz norm 
(Proposition~\ref{propalexinvariance}).   It is clear from the proof of Proposition~\ref{propconvergence}
that the convergence theorem applies for limits of two  continuous variables.  We can then take
limits as $x$ and $y$ tend to $\infty$ or $-\infty$
under the integral signs of $g\ast f$.  Note that
$$
g\ast f(x,y)=\intinf\intinf f(s,t)g(x-s,y-t)\,dt\,ds.
$$
And,
$$
\lim_{\substack{x\to\infty\\y\to\infty}}g(x-s,y-t)=\left\{\begin{array}{cl}
g(\infty,\infty), & (s,t)\not=(\infty,\infty)\\
g(-\infty,-\infty), & (s,t)=(\infty,\infty).
\end{array}
\right.
$$
As per Proposition~\ref{propbvline} we can ignore the value of the integrand in
$g\ast f(x,y)$ on two coordinate lines.  Hence, the limit of $f\ast g(x,y)$ as $x,y
\to\infty$ gives $F(\infty,\infty)g(\infty,\infty)$.  Similarly, for the other
cases. 
This also shows $f\ast g\in C(\Rbar^2)$.
Part (d) can also be proved with integration by parts.
(e) To show $g\ast h\in\hkbv(\Rbar^2)$ let $(a_i,b_i)\times (c_i,d_i)$ be disjoint intervals in $\Rbar^2$.
By dominated convergence and the Fubini--Tonelli theorem we have
\begin{align*}
&\sum\abs{g\ast h(a_i,c_i)+g\ast h(b_i,d_i)-g\ast h(a_i,d_i)-g\ast h(b_i, c_i)}\\
&\leq \intinf\intinf\sum\abs{g(a_i-x,c_i-y)
+g(b_i-x,d_i-y)-g(a_i-x,d_i-y)\\
&\qquad -g(b_i-x, c_i-y)}\abs{h(x,y)}\,dy\,dx.
\end{align*}
From this it follows that $V_{12}g\ast h\leq V_{12}g\norm{h}_1$.  Similarly,
$\norm{V_1g\ast h}_\infty\leq \norm{V_1g}_\infty\norm{h}_1$ and
$\norm{V_2g\ast h}_\infty\leq \norm{V_2g}_\infty\norm{h}_1$.  Also, 
$\norm{g\ast h}_\infty\leq\norm{g}_\infty\norm{h}_1$.  Hence, $g\ast h\in\hkbv(\Rbar^2)$.
Part (d) now shows $f\ast (g\ast h)\in C(\Rbar^2)$.  
To show $f\ast (g\ast h)=(f\ast g)\ast h$ requires
a change in order of integration and this is justified by Proposition~\ref{propfubini2}.
\end{proof}
When $g\in L^1(\R^2)$ the 
convolution is not directly defined as above when $g\in\hkbv(\Rbar^2)$.  However, the density
of $L^1(\R^2)$ in $\alextwo$ (Proposition~\ref{propdense}) lets us define the convolution
using a sequence of $L^1$ functions.
\begin{defn}\label{defnconvolutionacL1}
Let $f\in\alextwo$.  Let $\{f_n\}\subset L^1(\R^2)$ such that $\lim_{n\to\infty}
\norm{f_n-f}=0$.  For $g\in\hkbv(\Rbar^2)$ define $f\ast g$ as the unique distribution
in $\alextwo$ such that $\lim_{n\to\infty}\norm{f_n\ast g - f\ast g}=0$.
\end{defn}

To show this makes sense, let $\{f_n\}$, $f$ and $g$ be as in the definition.  Let
$x,y\in\R$.  Then, using the Fubini--Tonelli theorem,
$$
\int_{-\infty}^x\int_{-\infty}^y f_n\ast g(s,t)\,dt\,ds =
\intinf\intinf g(\xi,\eta)\int_{-\infty}^x\int_{-\infty}^y f_n(s-\xi,t-\eta)\,dt\,ds\,d\eta\,d\xi.
$$
It follows that $\norm{f_n\ast g}\leq \norm{f_n}\norm{g}_1$.  Hence, $\{f_n\ast g\}$ is a Cauchy
sequence in $\alextwo$ and therefore converges to a unique element of $\alextwo$.  This also shows
that $f\ast g$ is independent of the defining sequence $\{f_n\}$.

\begin{theorem}\label{theoremconvolutionacL1}
Let $f\in\alextwo$ and $g\in L^1(\R^2)$.  Define $f\ast g$ as in 
Definition~\ref{defnconvolutionacL1}.  Then (a) $f\ast g\in\alextwo$ and $\norm{f\ast g}\leq \norm{f} \norm{g}_1$.
(b) Let $h\in L^1(\R^2)$.  Then $(f\ast g)\ast h=f\ast (g\ast h)\in\alextwo$.
(c) Define $g_r(x,y)=r^{-2}g(r^{-1}x,r^{-1}y)$ for $r>0$.
Let $A=\intinf\intinf g_r(x,y)\,dy\,dx=\intinf \intinf g$.  Then $\norm{f\ast g_r-Af}\to 0$
as $r\to 0^+$.
\end{theorem}
\begin{proof}
Let $\{f_n\}\subset L^1(\R^2)$ such that $\norm{f_n-f}\to 0$.
(a) We have $\norm{f_n}\to\norm{f}$ and the inequality
$$
\norm{f_n\ast g}-\norm{f\ast g-f_n\ast g}\leq \norm{f\ast g}\leq \norm{f_n\ast g}+\norm{f\ast g-f_n\ast g}.
$$
Hence,
$$
\norm{f\ast g}=\lim_{n\to\infty}\norm{f_n\ast g}\leq \lim_{n\to\infty}\norm{f_n}\norm{g}_1=\norm{f}\norm{g}_1.
$$
(b) From the $L^1$ theory of convolutions it is known that $g\ast h\in L^1(\R^2)$.  For example,
\cite{folland}.  Then, by (a), $f\ast (g\ast h)\in\alextwo$.  And, $f\ast g\in\alextwo$ so by (a),
$(f\ast g)\ast h\in\alextwo$.
Hence, both $f\ast(g\ast h)$ and $(f\ast g)\ast h$ exist in $\alextwo$.  To show they are equal
note that convolutions are associative in $L^1(\R^2)$.  Therefore,
$$
0=\lim_{n\to\infty}\norm{f_n\ast(g\ast h)-f\ast(g\ast h)}=\lim_{n\to\infty}\norm{(f_n\ast g)\ast h-f\ast(g\ast h)}.
$$
And, $f_n\ast g\in\alextwo$ such that $\norm{f_n\ast g- f\ast g}\to 0$.
Therefore, $\norm{(f_n\ast g)\ast h-(f\ast g)\ast h}\to 0$.
It now follows that $f\ast(g\ast h)=(f\ast g)\ast h$.
(c) If suffices to prove that $\norm{f_n\ast g_r -Af_n}\to 0$.  Accordingly,
$$
f_n\ast g_r(s,t)-Af_n(s,t) 
=\intinf\intinf\left[ f_n(s-r\xi,t-r\eta)-f_n(s,t)\right]g(\xi,\eta)\,d\eta\,d\xi.
$$
We can change variables by Theorem~\ref{theoremchange}.  To find the Alexiewicz norm,
the above expression is integrated from $s=-\infty$ to $x$ and from $t=-\infty$ to $y$,
for some $(x,y)\in\R^2$.  In the integral with $f_n(s-r\xi,t-r\eta)$ the order of integration can be changed 
due to the Fubini--Tonelli theorem.  In the integral with $f_n(s,t)$ the order of integration can be changed
since the $(s,t)$ variables separate from the $(\xi,\eta)$ variables.
This then gives
\begin{align*}
&\left|\int_{-\infty}^x\int_{-\infty}^y\left[f_n\ast g_r(s,t)-Af_n(s,t)\right]\,dt\,ds\right|\\
&\leq \intinf\intinf\left|\int_{-\infty}^x\int_{-\infty}^y\left[f_n(s-r\xi,t-r\eta)-f_n(s,t)\right]
dt\,ds\right|
\abs{g(\xi,\eta)}\,d\eta\,d\xi\\
&\leq \intinf\intinf \norm{f_n(\cdot-r\xi,\cdot-r\eta)-f_n(\cdot,\cdot)}\abs{g(\xi,\eta)}\,d\eta\,d\xi.
\end{align*}
Dominated convergence and continuity in the Alexiewicz norm (Proposition~\ref{propalexinvariance})
allows us to take the limit $n\to\infty$ under the integral sign.
\end{proof}

\begin{example}
Part (c) of this theorem is useful for showing the solution of a differential equation takes on
initial or boundary values in the Alexiewicz norm.  For example, if 
$\Phi_z(x,y)=z(x^2+y^2+z^2)^{-3/2}/(2\pi)$ is the half-space Poisson kernel from 
Proposition~\ref{propseparable}, then $\lim_{z\to 0^+}\norm{f\ast\Phi_z-f}=0$.  Then the convolution
$u(x,y,z)=f\ast\Phi_z(x,y)$ satisfies the boundary condition $u=f$ in the Alexiewicz norm when 
$z\to 0^+$.  The partial derivatives of $\Phi_z$ are of bounded variation.  Proposition~\ref{propconvergence}
can then be used to show we can differentiate under the integrals 
and this shows
$u$ is harmonic in the half-space $(x,y,z)\in\R^2\times(0,\infty)$.
\end{example}

\section{The integral in $\Rbar^n$}\label{sectionRbarn}
Here we will briefly sketch out the differences between the integral in $\R^2$ and in $\R^n$.

We now let $D_n=\partial_1\partial_2\ldots\partial_n$ and define
\begin{eqnarray*}
\balexn & = & \{F\in C(\Rbar^n)\mid
F(x)= 0 \text{ if } x_i=-\infty \text{ for some } 1\leq i\leq n\}\\
\alexn & = & \{f\in \Dpn\mid f= D_nF \text{ for some } F\in\balexn\}.
\end{eqnarray*}
As before, primitives are unique.  It is convenient to use matrix notation to
define the integral over interval
$I=[a_{21},a_{11}]\times[a_{22},a_{12}]\times\ldots\times[a_{2n},a_{1n}]$ by
$$
\int_If = \int_{a_{21}}^{a_{11}}\int_{a_{22}}^{a_{12}}\cdots\int_{a_{2n}}^{a_{1n}}f=(-1)^n\!\!\!\sum_{i_1,\ldots,i_n\in\{1,2\}}(-1)^{i_1+\ldots+i_n}F(a_{i_11},a_{{i_2}2},
\ldots,a_{{i_n}n}).
$$
There are $2^n$ summands.
This formula can be proved with induction by writing iterated integrals.

In the proof of Proposition~\ref{propseparable} the half-space Poisson kernel in $\R^n$ is given in 
\cite[p.~145]{axler}.

Hardy--Krause variation in $\Rbar^n$ is defined in Definition~$6.5.2$ in \cite{lee}.
The integration by parts formula, due to J.~Kurzweil, is given in \cite{lee}, Theorem~6.5.9.
See also \cite{young}.  Various forms of the second mean value theorem are given
in \cite{lee} and \cite{young}.

J.~Mawhin has listed the coordinate transformations that map intervals to intervals and this
will give a change of variables theorem as in Theorem~\ref{theoremchange}.
See \cite{mawhin}.

\end{document}